\documentclass[a4paper,reqno,11pt,oneside]{amsart}
\usepackage{amsmath,geometry}
\usepackage{graphicx}
\usepackage{mathrsfs}
\usepackage{amssymb,amsmath, amsfonts, amsthm}
\usepackage{latexsym}
\usepackage{color} 
\usepackage[dvips]{epsfig}
\usepackage[colorlinks]{hyperref}
\newtheorem{theorem}{Theorem}[section]
\newtheorem{lemma}[theorem]{Lemma}
\newtheorem{proposition}[theorem]{Proposition}
\newtheorem{corollary}[theorem]{Corollary}

\newtheorem{remark}[theorem]{Remark}

\numberwithin{equation}{section}

\renewcommand{\(}{\left(}
\renewcommand{\)}{\right)}
\newcommand{\R}{\mathbb R}
\newcommand{\dys}{\displaystyle}

\DeclareMathOperator{\diverg}{div}
\newcommand{\eps}{\varepsilon}

 \title
{\sc Normalized concentrating solutions to  nonlinear elliptic problems}

\begin{document}

\author{Benedetta Pellacci}
\address[B. Pellacci]{Dipartimento di Matematica e Fisica,
Universit\`a della Campania  ``Luigi Vanvitelli'',  via A.Lincoln 5, 81100
Caserta, Italy.}
\email{benedetta.pellacci@unicampania.it}
\author{Angela Pistoia}
\address[A. Pistoia]{Dipartimento SBAI, Sapienza Universit\`a di Roma,
via Antonio Scarpa 16, 00161 Roma, Italy }
\email{angela.pistoia@uniroma1.it}
\author{Giusi Vaira}
\address[G. Vaira]{Dipartimento di Matematica e Fisica,
Universit\`a della Campania  ``Luigi Vanvitelli'',  via A.Lincoln 5, 81100
Caserta, Italy.}
\email{giusi.vaira@unicampania.it}
\author{Gianmaria Verzini}
\address[G. Verzini]{Dipartimento di Matematica, Politecnico di Milano, p.za Leonardo da Vinci 32,  20133 Milano, Italy}
\email{gianmaria.verzini@polimi.it}

\thanks{Research  partially
supported by the MIUR-PRIN-2015KB9WPT Grant: ``Variational methods, with applications to
problems in mathematical physics and geometry'',
by  MIUR-PRIN-2017JPCAPN\_003 Grant
``Qualitative and quantitative aspects of nonlinear PDEs''
by
``Gruppo Nazionale per l'Analisi Matematica, la Probabilit\`a e le loro
Applicazioni'' (GNAMPA) of the Istituto Nazionale di Alta Matematica
(INdAM)}


 \begin{abstract}
We prove the existence of solutions $(\lambda, v)\in \R\times H^{1}(\Omega)$ of the elliptic
problem
\begin{equation*}
\begin{cases}
 \dys -\Delta v+(V(x)+\lambda) v =v^{p}\ &\text{ in $ \Omega, $}
 \\
v>0,\qquad \dys \int_\Omega v^2\,dx =\rho.
\end{cases}
\end{equation*}
Any $v$ solving such problem (for some $\lambda$) is called a normalized solution, where the
normalization is settled in $L^2(\Omega)$. Here $\Omega$ is either the whole space $\mathbb R^N$
or a bounded smooth domain of $\mathbb R^N$, in which case we assume $V\equiv0$ and homogeneous
Dirichlet or Neumann boundary conditions. Moreover,
$1<p<\frac{N+2}{N-2}$ if $N\ge 3$ and $p>1$ if $N=1,2$. Normalized solutions appear in different contexts, such as the study of the Nonlinear Schr\"odinger equation, or that of quadratic ergodic Mean Field Games systems. We prove the existence of solutions concentrating at suitable points of $\Omega$ as the prescribed mass $\rho$ is either small (when $p<1+\frac 4N$) or large (when $p>1+\frac 4N$) or it approaches some critical threshold  (when $p=1+\frac 4N$).
\end{abstract}

\maketitle

\section{Introduction}
Let $\Omega$ be a smooth open domain in  $\mathbb R^N$, $V:\Omega\to\R$  and $\rho>0$.
We study the existence of solutions $(\lambda, v)\in \R\times H^{1}(\Omega)$ of the elliptic
problem
\begin{equation}\label{P0}
\begin{cases}
 \dys -\Delta v+(V(x)+\lambda) v =v^{p}\ &\text{ in $ \Omega, $}
 \\
v>0,\qquad \dys \int_\Omega v^2\,dx =\rho,
\end{cases}
\end{equation}
where $p\in\(1, 2^*-1\)$. Here the usual critical Sobolev exponent is $2^* = 2N/(N-2)$ if $N\ge 3$ and $2^*=+\infty$ if $N=1,2$. In particular we will face two different cases: either
$\Omega=\R^{N}$, or $\Omega$ is a bounded smooth domain; in the latter case, we will
assume $V\equiv0$ and associate with \eqref{P0} homogeneous Dirichlet or Neumann boundary
conditions. Any $v$ solving \eqref{P0} (for some $\lambda$) is called a
\emph{normalized solution}, where the normalization is settled in $L^2(\Omega)$.

\subsection{Motivations}

Normalized solutions to semilinear elliptic problems are investigated in different applied models. One main, well-established motivation comes from the study of solitary waves to
time-dependent nonlinear  Schr\"{o}dinger equations (NLSE). For concreteness, let us consider the following NLSE for the time dependent, complex valued wave function $\Phi$:
\begin{equation}\label{eq:NLSEt}
i\partial_t \Phi + \Delta \Phi - V(x)\Phi + |\Phi|^{p-1}\Phi = 0,\qquad x\in\Omega,\ t\in\R.
\end{equation}
In this context, either $\Omega = \R^N$, or $\Omega$ can be a bounded domain, in which case homogeneous Dirichlet boundary conditions are imposed, to approximate an infinite well potential (i.e. $V(x) \equiv +\infty$ in $\R^N\setminus\Omega$). As it is well known \cite{Cazenave2003}, solutions to \eqref{eq:NLSEt} conserve, at least formally, the energy $E(\Phi)$ and the mass $Q(\Phi)$, where
\[
E(\Phi) = \frac12\int_\Omega |\nabla\Phi|^2 + \frac12\int_\Omega V(x)|\Phi|^2 - \frac{1}{p+1}\int_\Omega |\Phi|^{p+1},
\qquad
Q(\Phi) = \int_\Omega |\Phi|^2.
\]
Solitary wave solutions to \eqref{eq:NLSEt} are obtained imposing the \textit{ansatz}  $\Phi(x,t) = e^{i\lambda t} v(x)$, where the real constant $\lambda$ and the real valued function $v$ satisfy
\begin{equation}\label{eq:NLSE}
-\Delta v + (V(x) + \lambda) v = |v|^{p-1}v
\end{equation}
in $\Omega$, with suitable boundary conditions. Now, two  point of view can be adopted.

On the one hand, one can choose a fixed value of $\lambda$, searching for solutions $v$ of \eqref{eq:NLSE}. This can be done using either topological methods, such as fixed point theory or the
Lyapunov-Schmidt reduction, or variational ones, looking for critical points of the associated
action functional $J(v) = E(v) + \lambda Q(v) /2$. This point of view has been widely adopted in
the last decades, the related literature is huge, and we do not even try to summarize it here.

On the other hand, one can consider also $\lambda$ as part of the unknown. In this case it is
quite natural to fix the value $Q(v)$, so that one is led to consider normalized solutions. The
variational framework to treat this problem consist in  searching for critical points of the energy $E$, constrained to the Hilbert manifold $M_\rho = \{v:Q(v)=\rho\}$. In this way,
$\lambda$ plays the role of a Lagrange multiplier. Notice that, in the simplest case $\Omega = \R^N$, $V\equiv0$, the problem
\begin{equation}\label{eq:scalingpb}
 \begin{cases}
 \dys -\Delta v+\lambda  v =v^{p}\ &\text{ in $ \R^N, $}
 \\
v>0,\quad \dys \int_\Omega v^2\,dx =\rho,
\end{cases}
\end{equation}
can be completely solved by scaling, at least when dealing with positive $v$. More precisely, in the subcritical range $1<p<2^*-1$, let us denote with $U$ the unique radial solution (depending on $p$) to
\begin{equation}\label{pblim}
 -\Delta U + U =U^p,\qquad U\in H^1(\mathbb R^N),\ U>0\ \mbox{in}\  \mathbb R^N,
\end{equation}
having mass
\begin{equation}\label{so}
2\sigma_0 = 2\sigma_0(p) :=\int_{\mathbb R^N} U^2(x)\,dx>0.
\end{equation}
It is well know that any positive solution in $H^1(\R^N)$ of $-\Delta v + v = v^p$ is a translated copy of $U$. Therefore we obtain that $(\lambda , v)$ solves  \eqref{eq:scalingpb}
if and only if
\[
\lambda>0, \qquad v(x) = \lambda^{\frac{1}{p-1}} U(\lambda^{\frac12}x),\qquad \rho =
\lambda^{\frac{2}{p-1}-\frac{N}{2}} \cdot 2\sigma_0.
\]
As a consequence, \eqref{eq:scalingpb} is solvable for every $\rho$ whenever $\dfrac{2}{p-1}-\dfrac{N}{2}\neq 0$ (and the solution is unique up to translations). The complementary case
corresponds to the so-called \emph{mass critical} (or  $L^2$-critical) exponent:
\[
p = 1 + \frac{4}{N}
\qquad\implies\qquad
\text{\eqref{eq:scalingpb} is solvable iff } \rho = 2\sigma_0
\]
(with infinitely many solutions, one for every $\lambda>0$). As we will see,
 on a general ground, for the mass critical exponent the existence of normalized solutions becomes strongly
unstable. Incidentally, the criticality of such exponent has repercussions also in other aspects
of \eqref{eq:NLSEt}, related to dynamical issues (orbital stability, blow-up) also in connection
with the exponents appearing in the Gagliardo-Nirenberg inequality, see \cite{Weinstein1983, Cazenave2003}.

When scaling is not allowed, the existence of normalized solutions becomes nontrivial, and many
techniques developed for the case with fixed $\lambda$ can not be directly adapted to this
framework. Also for this reason, the literature concerning normalized solutions is far less broad: after the paper by Jeanjean \cite{MR1430506} in 1997, concerning autonomous equations on
$\R^N$ with non-homogeneous nonlinearities, only recently an increasing number of papers deal with this subject. Different lines of investigation include, for instance, NLS equations and systems on $\R^N$\cite{MR3009665,MR3539467,BellazziniJeanjean,MR3639521,MR3638314,MR3777573,BellazziniGeorgievVisciglia,0951-7715-31-5-2319,MR3895385,2018arXiv181100826S,2019arXiv190102003S}, on bounded domains
\cite{ntvAnPDE,ntvDCDS,MR3689156,MR3918087} or on quantum graphs
\cite{AST1,MR3494248,AST2,MR3758538,pierotti2019local}.

More recently, normalized solutions have been considered also in connection with Mean Field Games (MFG) theory, which has been introduced by seminal papers of Lasry and Lions \cite{jeux1,jeux2,LasryLions} and of Caines, Huang, Malham\'e \cite{HCM}.
Such theory models the behavior of a large number of indistinguishable rational agents, each aiming at minimizing some common cost. In the ergodic case, when the cost is of
long-time-average type, the distribution of the players becomes stationary in time. For our aims, we focus on ergodic MFG with quadratic Hamiltonian and power-type, aggregative interaction.
The reason of this choice is that in this case, contrary to the general one, the MFG system can be reduced to \eqref{P0} by a change of variable. In the setting we want to describe, the state of a typical agent is driven by the controlled stochastic differential equation
\[
d X_t = -a_t dt + \sqrt{2 \nu} \, d B_t,
\]
where $a_t$ is the controlled velocity and $B_t$ is a Brownian motion, with initial state provided by the random variable $X_0$. The player chooses $a_t$ in such a way to minimize the  cost
\[
\mathcal{J}(X_0, a) = \liminf_{T \rightarrow \infty} \frac{1}{T} \int_0^T \mathbb{E}  \left[\frac{|a_t|^2}{2} + V(X_t) - \alpha m^q(X_t)\right] dt,
\]
where $q>0$, $V$ is a given potential and $m(x)$ denotes the (observed) density of the players at $x\in\Omega$.
As time $t\to+\infty$, the
distribution law of $X_t$ converges to a measure having density $\mu=\mu(x)$, independent of $X_0$, and at Nash
equilibria of the game the densities $\mu$ and $m$ coincide. From the PDE viewpoint, such equilibria are described by the following elliptic system, which couples a
Hamilton-Jacobi-Bellman equation for $u$ and a Kolmogorov equation for $m$, which has to satisfy also a normalization in $L^1(\Omega)$:
\begin{equation}\label{MFG}
\begin{cases}
- \nu \Delta u(x) + \frac{1}{2} {|\nabla u(x)|^2} = \lambda + V(x) -\alpha m^q(x) &
\text{in $\Omega$} \\
- \nu \Delta m(x) - \diverg (m(x) \, \nabla u(x)) = 0  & \text{in $\Omega$}  \\
\dys \int_\Omega m dx = 1,\qquad m>0.
\end{cases}
\end{equation}
Here the unknown $\lambda$ gives, up to a change of sign, the average cost, $\nabla u$ provides
an optimal control,  and $m$ is the stationary population density of agents playing with optimal
strategy. As we mentioned, we deal with the aggregative case, i.e. $\alpha>0$: indeed, in such
case, the individual cost $\mathcal{J}$ is decreasing with respect to $m$, and the agents are
attracted to crowded regions \cite{MR3532395,Gomes2018}. If we suppose that $\Omega$ is bounded and that its boundary $\partial \Omega$ acts as a reflecting barrier on the state $X_t$,
then \eqref{MFG} is naturally complemented with Neumann boundary conditions, both for $u$ and $m$
\cite{MR3333058}. Alternatively, one can consider \eqref{MFG} on $\Omega = \R^N$
\cite{MR3864209}. As we mentioned, the specific choice of the quadratic Hamiltonian
$H(p) = |p|^2/2$
allows to use the Hopf-Cole transformation \cite{LasryLions} in order to reduce \eqref{MFG} to a single PDE. Indeed, defining
\begin{equation}\label{eq:hopfcole}
v^2(x) := \alpha^{1/q} m(x) = c e^{-u(x)/\nu},
\end{equation}
for a suitable normalizing constant $c$, then $v$ solves
\[
\begin{cases}
 \dys -2\nu^2\Delta v+(V(x)+\lambda) v =v^{2q+1}\ &\text{ in $ \Omega, $}
 \\
v>0,\qquad \dys \int_\Omega v^2\,dx = \alpha^{1/q},
\end{cases}
\]
which reduces to \eqref{P0} by choosing $\nu = \sqrt2/2$, $p=2q+1$, $\rho = \alpha^{1/q}$.

\subsection{Main results}

A common feature of the papers listed above, both in the NLS and in the MFG case, is that
they use a variational approach: normalized solutions are found either as minimizers
or as saddle points of a suitable energy ($E$ in the NLS case) on the mass constraint.
Up to our knowledge, only few results about normalized solutions exploit non-variational
techniques: in particular, we refer to \cite{MR3660463}, where bifurcation techniques are applied
to a quadratic multi-population MFG system.

In the present paper we propose a first approach to problem \eqref{P0} based on the
Lyapunov-Schmidt reduction. Indeed, setting
\begin{equation}\label{eq:v<->u}
\varepsilon:=\lambda^{-\frac12},\qquad
u:=\varepsilon^{{\frac2{p-1}}}v,
\end{equation}
problem \eqref{P0} turns to be equivalent to
\begin{equation}\label{p}
\begin{cases}
 \dys -\varepsilon^2\Delta u+(\varepsilon^2V(x)+1) u =u^{p}\ &\text{ in $ \Omega, $}
 \\
u>0,\quad \dys \varepsilon^{-\frac{4}{p-1}}\int_\Omega u^2\,dx =\rho.
\end{cases}
\end{equation}
We treat \eqref{p} as a singularly perturbed problem, looking for solutions
$(\varepsilon,u)$, with $\varepsilon$ sufficiently small, via a Lyapunov-Schmidt reduction. By \eqref{eq:v<->u}, these correspond to solutions $(\lambda,v)$ of the original problem \eqref{P0}, with $\lambda$ large. As a matter of fact, this strategy will work for selected ranges of $\rho$, depending on $p$.

As an important advantage of our approach we are able to describe the
asymptotic profile of the solutions we find, in terms  of the solution $U\in H^1(\mathbb R^N)$
of problem \eqref{pblim}. More precisely, we find solutions which are approximated by a suitable scaling of $U$, concentrated at suitable points.

Roughly speaking, we say that a family $v=v_\rho$ of solutions of \eqref{P0}, indexed on $\rho$,
\emph{concentrates} at some point $\xi_0 \in \overline{\Omega}$ as
$\rho\to\rho^*\in[0,+\infty]$ if
\begin{equation}\label{eq:qualit_develop}
v_\rho (x) = \eps_\rho^{-\frac{2}{p-1}} U \(\frac{x-\xi_\rho}{\varepsilon_\rho}\)+R_\rho(x),
\end{equation}
where, as $\rho\to\rho^*$, $\eps_\rho\to 0$, $\xi_\rho \to \xi_0$, and the remainder $R_\rho$ is a lower order term, in some suitable sense.

About the point of concentration $\xi_0$, we deal with three different cases, namely:
\begin{enumerate}
\item\label{it:br} $\Omega$ bounded, $V\equiv 0$, Neumann boundary conditions, in which case $\xi_0\in\partial\Omega$ is a non-degenerate critical point of the mean
curvature of the boundary $\partial\Omega$;
\item\label{it:int} $\Omega$ bounded, $V\equiv 0$, either Dirichlet or Neumann boundary conditions, in which case $\xi_0\in \Omega$ is the maximum point   of the distance function from the $\partial\Omega$;
\item\label{it:sc} $\Omega=\R^N$, in which case $\xi_0\in\R^N$ is a non-degenerate critical point of $V$.
\end{enumerate}
To illustrate the kind of results we obtain, we provide here a qualitative, incomplete statement
concerning each case. Let us start with the boundary concentration case \eqref{it:br} which will be treated in  Section \ref{sub:br}  (see Theorems \ref{main1},
\ref{main1crit}).
 \begin{theorem}\label{thm:introN}
Let us consider \eqref{P0}, with $\Omega$ bounded and $V\equiv 0$,
associated with Neumann boundary conditions. Let  $\xi_0\in \Omega$ be a non-degenerate critical point of
the mean curvature $ H$ of the boundary $\partial\Omega$. There exists $
\rho_0 = \rho_0(p,\Omega)>0$ such that:
\begin{itemize}
\item if $1<p<1+\dfrac{4}{N}$ there exist solutions $v_\rho$ for every $\rho>\rho_0$, concentrating at $\xi_0$ as $\rho\to+\infty$;
\item if $1+\dfrac{4}{N}<p<2^*-1$ there exist solutions $v_\rho$ for every $0<\rho<\rho_0$, concentrating at $\xi_0$ as $\rho\to0$;
\item if $p=1+\dfrac{4}{N}$, $H(\xi_0)\not=0$ and \eqref{sigma1} holds true, there exist solutions $v_\rho$ for every $\sigma_0 - \rho_0<\rho< \sigma_0$ or
$\sigma_0<\rho< \sigma_0 + \rho_0$
depending on the sign of the mean curvature at $\xi_0$; in both cases, $v_\rho$ concentrates at $\xi_0$ as $\rho\to \sigma_0$ ($\sigma_0$ being defined in \eqref{so}).
\end{itemize}
\end{theorem}
Theorem \ref{thm:introN} can be immediately translated to the MFG system \eqref{MFG}. Recalling
\eqref{eq:hopfcole}, in this case the leading parameter is $\alpha$ and the concentration of
the density $m_\alpha$ is intended as
\[
m_\alpha (x) = (\alpha\eps_\alpha^2)^{-\frac{1}{q}} U^2 \(\frac{x-\xi_\alpha}{\varepsilon_\alpha}\)+R_\alpha(x).
\]
\begin{corollary}\label{coro:intro1}
Let us consider the MFG system \eqref{MFG}, with $\nu = \sqrt2/2$, $\Omega$ bounded and $V\equiv 0$, associated with Neumann boundary conditions. Let $\xi_0\in \Omega$ be a non-degenerate critical point of
the mean curvature $ H$ of the boundary $\partial\Omega$. There exists $\alpha_0 = \alpha_0(p,\Omega)>0$ such that:
\begin{itemize}
\item if $0<q<\dfrac{2}{N}$ there exist solutions $m_\alpha$ for every $\alpha>\alpha_0$,
concentrating at $\xi_0$ as $\alpha\to+\infty$;
\item if $\dfrac{2}{N}<q<\dfrac{2^*-2}{2}$ there exist solutions $m_\alpha$ for every $0<\alpha<\alpha_0$, concentrating at $\xi_0$ as $\alpha\to0$;
\item if $q=\dfrac{2}{N}$,  $H(\xi_0)\not=0$ and \eqref{sigma1} holds true, there exist solutions $m_\alpha$ either for every $\sigma_0^q - \alpha_0<\alpha< \sigma_0^q$ or
$\sigma_0^q<\alpha< \sigma_0^q + \alpha_0$, depending on the sign of the mean curvature at
$\xi_0$; in both cases, $m_\alpha$ concentrates at $\xi_0$ as $\alpha\to\sigma_0^q$.
\end{itemize}
\end{corollary}
Since \eqref{MFG} with $\alpha>0$ entails an aggregative interaction between the players,
concentrating solutions are somehow expected. In \cite{MR3864209}, concentrating solutions were
obtained for more general, non-quadratic MFG, in the mass subcritical case, by variational
methods. Our results are reminiscent of those obtained in \cite[Thm. 1.1]{MR3532395}.

Let us state our results concerning the interior concentration case \eqref{it:int} which will be treated in  Section \ref{sub:int} (see Theorems \ref{main2},
\ref{main2critico}).
\begin{theorem}\label{thm:intro}
Let us consider \eqref{P0}, with $\Omega$ bounded and $V\equiv 0$, associated with either homogeneous Dirichlet boundary conditions or Neumann ones. Let $\xi_0\in \Omega$ be the maximum point of the distance function from the $\partial\Omega$. There exists $\rho_0 = \rho_0(p,\Omega)>0$ such that:
\begin{itemize}
\item if $1<p<1+\dfrac{4}{N}$ there exist solutions $v_\rho$ for every $\rho>\rho_0$, concentrating at $\xi_0$ as $\rho\to+\infty$;
\item if $1+\dfrac{4}{N}<p<2^*-1$ there exist solutions $v_\rho$ for every $0<\rho<\rho_0$, concentrating at $\xi_0$ as $\rho\to0$;
\item if $p=1+\dfrac{4}{N}$, there exist solutions $v_\rho$ for every $2\sigma_0 - \rho_0<\rho<2\sigma_0$ in the Dirichlet case, and for every
$2\sigma_0<\rho< 2\sigma_0 + \rho_0$ in the Neumann one; in both cases, $v_\rho$ concentrates at $\xi_0$ as $\rho\to2\sigma_0$.
\end{itemize}
\end{theorem}
\begin{corollary}\label{coro:intro2}
Let us consider the MFG system \eqref{MFG}, with $\nu = \sqrt2/2$, $\Omega$ bounded and $V\equiv 0$, associated with Neumann boundary conditions. Let $\xi_0\in \Omega$ be the maximum point of the distance function from the $\partial\Omega$. There exists $\alpha_0 = \alpha_0(p,\Omega)>0$ such that:
\begin{itemize}
\item if $0<q<\dfrac{2}{N}$ there exist solutions $m_\alpha$ for every $\alpha>\alpha_0$, concentrating at $\xi_0$ as $\alpha\to+\infty$;
\item if $\dfrac{2}{N}<q<\dfrac{2^*-2}{2}$ there exist solutions $m_\alpha$ for every $0<\alpha<\alpha_0$, concentrating at $\xi_0$ as $\alpha\to0$;
\item if $q=\dfrac{2}{N}$, there exist solutions $m_\alpha$ for every $(2\sigma_0)^q<\alpha< (2\sigma_0)^q+ \alpha_0$, concentrating at $\xi_0$ as $\alpha\to(2\sigma_0)^q$.
\end{itemize}
\end{corollary}
Analogous results holds also for MFG systems with Dirichlet boundary conditions.

Finally, we   state our results concerning the last case \eqref{it:sc} which will be treated in  Section \ref{sub:sc}  (see Theorems \ref{main3},
\ref{th:main3crit}).
\begin{theorem}\label{thm:introR}
Let us consider \eqref{P0}, with $\Omega=\R^{N}$. Let $\xi_0\in \Omega$ be a non-degenerate critical point of the potential $V$. There exists $\rho_0 = \rho_0(p,V)>0$ such that:
\begin{itemize}
\item if $1<p<1+\dfrac{4}{N}$ there exist solutions $v_\rho$ for every $\rho>\rho_0$, concentrating at $\xi_0$ as $\rho\to+\infty$;
\item if $1+\dfrac{4}{N}<p<2^*-1$ there exist solutions $v_\rho$ for every $0<\rho<\rho_0$, concentrating at $\xi_0$ as $\rho\to0$;
\item if $p=1+\dfrac{4}{N},$ $\Delta V(\xi_0)\not=0$ and     \eqref{ipo} holds true, then there exist solutions $v_\rho$ for every $2\sigma_0 - \rho_0<\rho< 2\sigma_0$ or
$2\sigma_0<\rho< 2\sigma_0 + \rho_0$ depending on the sign of  $\Delta V(\xi_0)$; in both cases, $v_\rho$ concentrates at $\xi_0$ as $\rho\to \sigma_0$.\end{itemize}
\end{theorem}
Again, a natural counterpart of the above result can be written in the setting of MFG systems with potentials on $\R^N$.

As we mentioned, the proof of our results consists in rephrasing problem \eqref{P0}
into the singularly perturbed problem \eqref{p} whose solutions can be found via the
 well known Ljapunov-Schmidt procedure.
We shall omit many details on this procedure because they can be found, up to some
minor modifications, in the literature. We only compute what cannot be deduced from known results.

When $p\neq 1+\frac4N$ our results provide an almost
complete picture only assuming the non-degeneracy of a critical point
$\xi_{0}$. Indeed, under this assumption we can produce solutions concentrating
at $\xi_{0}$, provided either the mass is large, in the sub-critical regime,
or small in the super-critical one; moreover, we can also exhibit
exact asymptotics both  for the concentration parameter
$\eps_\rho$ and for the remainder $R_\rho$ in equation \eqref{eq:qualit_develop}.

On the other hand, the study of the critical regime, i.e. $p=1 + \frac 4N$, needs new  delicate  estimates
of the error term whose proof  requires a lot of technicalities. This affects different aspects. First, we can construct
concentrating solutions only when the mass is close to the threshold
value $\sigma_{0}$ (defined in \eqref{so}); however this appears
as a natural obstruction  that has already been observed in the literature (see \cite{ntvAnPDE,MR3689156}). What is more relevant is that we can prove our result
without any further assumption only in the case of interior concentration
(see Theorem \ref{thm:intro}), while we need additional hypotheses both
in cases (1) and (3) (see Theorems \ref{thm:introN}, \ref{thm:introR}).
As a matter of fact, in these latter situations we assume that mean curvature
of the boundary or the laplacian of the potential $V$ cannot vanish at the
concentration point $\xi_{0}$; furthermore, we also suppose \eqref{sigma1},
or \eqref{ipo} which appear difficult to be checked as they concern global
information involving not explicit solutions to linear problems (see
\eqref{pW} and \eqref{ipo2}). Actually, we succeeded in verifying \eqref{ipo}
only in the one dimensional case (see Remark \ref{N=1}), but we think
that they hold in every dimension and it would be extremely interesting
to provide a proof for them.

The critical case $p=1+\frac4N$ also presents important difficulties in
the determination of the exact asymptotic of $\varepsilon_{\rho}$ and the remainder  term $R_{\rho}$: we can give this kind of precise information,
as in the sub- and super-critical regime, only in case \eqref{it:int} and  for $N=1$
(see Remark  \ref{rem:DirN=1}).

Concerning  the interval of allowed $L^{2}$ masses in the critical case,
let us notice that the existence of solutions concentrating at $\xi_0$ is
established when the mass approaches the critical values $\sigma_0$ or $
 2\sigma_0$ (see \eqref{so}) either from below or from above.
We strongly believe that our results are sharp, in the class of single-peak concentrating solutions.
Let us make our claim more precise with a couple of examples.
In Theorem \ref{thm:intro} when $\Omega$ is a ball we prove   that the Dirichlet
problem and the Neumann problem have a solution concentrating at the origin
provided the mass approaches $2\sigma_0$ from below and from above,
respectively. We conjecture that  these solutions do not exist when the mass
approaches $2\sigma_0$ from above or from below, respectively (actually, in the Dirichlet case,
this is known to be true in the class of positive solutions, see \cite[Thm. 1.5]{ntvAnPDE}).
Theorem  \ref{thm:introR} in the $1-$dimensional case (see also
Remark \ref{N=1}) states the existence of a solution, concentrating at  a
non-degenerate minimum   or   maximum point of the potential $V$  when the mass
approaches
$2\sigma_0$ from below or from above, respectively.
Again we strongly believe that these kind of solutions do
not exist when the mass approaches $2\sigma_0$ from above or from below, respectively.

As our interest in this article focuses in the existence
of normalized concentrating solutions, we have considered only
the simplest case of concentration; however, using similar ideas, it should
be possible to build solutions concentrating at multiple points; in the critical case, this
should provide multi-peak solutions having mass which
approaches integer multiples of the critical value $2\sigma_0$ ($\sigma_0$ in the case
of boundary concentration).

However, single-peak solutions are more interesting when looking for orbitally stable
standing waves of NLSE. Indeed, in this research line, a key information relies on proving
that the Morse index of the normalized solution is $1$. Actually, we are able to provide this information in dependence of the Morse index of the point $\xi_{0}$ itself, as pointed out in Remarks \ref{rm:morsebordo}, \ref{rm:morseint} and \ref{rm:morsesc}.

Finally, in this paper we always consider Sobolev sub-critical powers.
The case $p=2^{*}-1$ with boundary conditions has been recently studied in \cite{MR3918087}
and we believe that our approach, together with results obtained by Adimurthi and Mancini in \cite{am}, could be used to tackle boundary concentration for the Neumann problem. In particular,
this should be possible at non-degenerate critical points of the mean curvature, having positive mean curvature. On the other hand, global or local Pohozaev's identities imply non-existence of solutions of \eqref{p}, for $\varepsilon$ small, for the Dirichlet problem on star-shaped domains
\cite{bn} and for the Schr\"odinger equation for suitable potentials \cite{cp}.

The paper is organized as follows.   Section \ref{sec:bdd} is devoted to study the problem on bounded domains. In particular in Section \ref{sub:br} we build solutions concentrating at suitable boundary points for the Neumann problem, while in Section \ref{sub:int} we build  solutions concentrating at  the most centered point of the domain for both Neumann and Dirichlet problems. The Schr\"{o}dinger equation defined in the whole space is studied in Section \ref{sub:sc}, where solutions concentrating at suitable critical points of the potential $V$ are constructed.

\section{The   problem on a bounded domain}\label{sec:bdd}

  In this section we consider Problem \eqref{p} in a bounded domain $\Omega$
  in  $\mathbb R^N$, with either Neumann or Dirichlet boundary conditions.

\subsection{Boundary concentration}\label{sub:br}

 In this subsection we will study Problem \eqref{p} in a bounded domain
 $\Omega$ with homogeneous Neumann boundary conditions, focusing our
attention on the existence of solutions concentrating at some point on the boundary of $\Omega$. Our Theorems will rely on some well known results
due to  Li \cite{yyl} and Wei \cite{w} concerning the existence of solutions to
 the  following singularly perturbed Neumann problem
  \begin{equation}\label{pn}
  \begin{cases}
  -\varepsilon^2\Delta u+  u = u^p & \text{in}\ \Omega,
 \\
 u>0 & \text{in}\ \Omega,
  \\
 \partial_\nu u=0 & \text{on}\ \partial\Omega
\end{cases}
\end{equation}
 as $\varepsilon$ is small enough.\\
 We will only consider the case  $N\geq 2$, because when $N=1$ solutions concentrating on the boundary point of an interval can be found by reflection as solutions concentrating on an interior point as we will show in the next section.

For future convenience,
let us introduce some notations. Given a point $\xi_0\in\partial\Omega,$
without loss of generality we can assume that $\xi_0 =0$  and $x_N=0$ is the tangent plane of $\partial\Omega$ at $\xi_0$ and $\nu(\xi_0)=(0, 0, \ldots, -1)$. We also assume that $\partial\Omega$ is given by $x_N=\psi(x')$ where $\psi$ is a real and smooth function defined in
$\displaystyle
\left\{x'\in\mathbb R^{N-1}\ :\ |x'|<\eta\right\}
$
for some $\eta>0$ such that
\begin{equation}\label{eq:defpsi}
\psi(x'):=\frac 12\sum_{j=1}^{N-1}\kappa_j x_j^2+O(|x'|^3)\quad \hbox{if}\ |x'|<\eta.
\end{equation}
Here $\kappa_j=\kappa_j(\xi_0)$ are the principal curvatures and
  $H(\xi_0)=\frac{1}{N-1}\sum_{j=1}^{N-1}\kappa_j(\xi_0)$ is the mean curvature at the boundary point $\xi_0$.

We will denote with $U$ the $H^1(\mathbb R^N)$ solution to \eqref{pblim}, enjoying the following properties
\begin{equation}\label{groundstate}
\left\{
\begin{aligned}
&U(x)=U(|x|)\qquad \forall\,\, x\in\mathbb R^N\\
&U'(r)<0\,\,\ \forall\,\, r>0,\,\,\, U''(0)>0\\
&\lim_{r\to+\infty} r^{\frac{N-1}{2}}e^r U(r)=\mathfrak c>0;\qquad \lim_{r\to+\infty}\frac{U'(r)}{U(r)}=-1.
\end{aligned}\right.
\end{equation}

The following statement collects the facts, that we will use, concerning the existence of concentrating solutions for Problem \eqref{pn}.

 \begin{theorem}[\cite{yyl,w}]\label{li-wei}
 Let $\xi_0\in\partial\Omega$ be a non-degenerate critical point of the mean curvature of the boundary $\partial\Omega.$ Then
there exists $\varepsilon_0>0$ such that for any $\varepsilon\in(0,\varepsilon_0)$ there exists a solution $u_\varepsilon$ to \eqref{pn}
 which concentrates at the point $\xi_0$ as $\varepsilon\to0.$ More precisely
 \begin{equation}\label{refined}
u_\varepsilon(x)= U\({x-\xi_\varepsilon\over\varepsilon}\)+\varepsilon V_{\xi_0}\({x-\frac{\xi_\varepsilon}{\varepsilon}}\)+\psi_\varepsilon(x)
\end{equation}
where $\xi_\varepsilon\in\partial\Omega$ and
\begin{equation}\label{points}
{\xi_\varepsilon-\xi_0\over\varepsilon}\to 0\ \hbox{as}\ \varepsilon\to0,
\end{equation}

the reimander term $\psi_\varepsilon$ satisfies
\begin{equation}\label{restoref}
\|\psi_\varepsilon\| _{H^1_\varepsilon(\Omega)}:=
\left[\int\limits_{\Omega}\(\varepsilon^2|\nabla \psi_\varepsilon|^2+  \psi_\varepsilon^2\)dx\right]^{1/2}\!\!\!=
 \mathcal O\(\varepsilon^{\min\{2,p\}+\frac N2}\).
\end{equation}
The function $V_{\xi_0}\in H^1(\mathbb R^N)$ solves the linear problem
\begin{equation}\label{pV}
\left\{
\begin{aligned}&
-\Delta V_{\xi_0}+V_{\xi_0}-pU^{p-1} V_{\xi_0}=0\ \hbox{in}\ \mathbb R^N_+,\\
& {\partial V_{\xi_0}\over \partial y_N}(y',0)=\frac 12 {U'(|y'|,0)\over |y'|}\sum\limits_{i=1}^{N-1}\kappa_i(\xi_0) y_i^2=
\frac 12 \sum\limits_{i=1}^{N-1}\kappa_i(\xi_0){\partial U\over\partial y_i}(y',0)y_i\ \hbox{on}\ \partial\mathbb R^N_+
\end{aligned}\right.
\end{equation}
and it is given by
\[
V_{\xi_0}(y)=\frac 12 \sum\limits_{i=1}^{N-1}\kappa_i(\xi_0) \left(
{\partial U\over\partial y_i}(y)y_iy_N+W_i(y)\right)
\]
where $W_i$ solves
\begin{equation}\label{pW}
\left\{
\begin{aligned}&
-\Delta W_i+W_i-pU^{p-1} W_i=2\left(y_N\partial_{ii}U+y_i\partial_{iN}U\right)\ \hbox{in}\ \mathbb R^N_+,\\
& {\partial W_i\over \partial y_N}(y',0)=0\ \hbox{on}\ \partial\mathbb R^N_+.
\end{aligned}\right.
\end{equation}
 \end{theorem}
\begin{remark}
Note that, using the invariance by symmetry of $\Delta$, it is immediate to check that
\[
W_i (y_1, \dots,y_i,\dots,y_N)= W_1 (y_i, \dots,y_1,\dots,y_N).
\]
\end{remark}
 Now, let us consider the Neumann problem with prescribed $L^2-$norm
  \begin{equation}\label{pnl}
    \begin{cases}
  -\varepsilon^2\Delta u+  u = u^p & \text{in}\ \Omega,
 \\
 u>0 & \text{in}\ \Omega,
  \\
 \partial_\nu u=0 & \text{on}\ \partial\Omega
\\
\varepsilon^{-{4\over p-1}}\int\limits_\Omega u^2 =\rho.
\end{cases}
\end{equation}
 Our first result concerns the existence of a solution of Problem \eqref{pnl} in the sub- and super-critical regime.
 \begin{theorem}\label{main1}
 Let $\xi_0\in\partial\Omega$ be a non-degenerate critical point of the mean
curvature of the boundary $\partial\Omega.$ Suppose that
$ p\neq \frac4N+1$ and take $\sigma_{0}$ as in \eqref{so}. The following conclusions hold
\begin{itemize}
\item[(i)] If $p<\frac 4N+1$ there exists $R>0$ such that for any $\rho>R$
Problem \eqref{pnl}  has a solution $(\Lambda_\rho, u_\rho)$
for  $\varepsilon:=\(\Lambda_\rho\rho\)^{(p-1)\over
(p-1)N-4}$, with
$\Lambda_\rho \to {1\over  {\sigma_0}}$ and $u_\rho$  concentrating at the point $ \xi_0$  as $\rho\to\infty$.
\item[(ii)]  If $p>\frac 4N+1$ there exists $r>0$ such that for any $\rho<r$
Problem  \eqref{pnl}   has a solution $(\Lambda_\rho, u_\rho)$
for  $\varepsilon:=\(\Lambda_\rho\rho\)^{(p-1)\over
(p-1)N-4}$, with
 $\Lambda_\rho  \to {1\over  {\sigma_0}}$ and $u_\rho$   concentrating at the point $
\xi_0$ as $\rho\to0$.
 \end{itemize}
 \end{theorem}
 \begin{proof}
In order to apply Theorem \ref{li-wei} we have to reduce the existence of solutions to Problem \eqref{pnl} with variable but prescribed $L^2-$norm to the existence of solutions to Problem \eqref{pn} where the parameter $\varepsilon$ is small.
Let us choose
 \begin{equation}\label{la}
 \varepsilon^{-{4\over p-1}+N }=\Lambda  \rho \ \hbox{with}\ \Lambda=\Lambda(\rho)\in\left[\frac{1}{2 \sigma_0} ,\frac2{ \sigma_0}\right]
 \end{equation}
is to be chosen later and  where  $\sigma_0$ is defined in \eqref{so}.
Note that $\varepsilon\to0$ if and only if either
  $ p<\frac 4N +1$ and $\rho\to\infty$ or
  $  p>\frac 4N +1$ and $\rho\to0.$

Theorem \ref{li-wei} implies that for any $\Lambda$ as in \eqref{la}, there exists either $R>0$ or $r>0$ such that for any $\rho>R$ or $\rho<r$ problem \eqref{pn} has a solution $u_{\varepsilon}$ as in \eqref{refined} such that $\varepsilon$ satisfies \eqref{la}.
 Now, we have to choose the free parameter $\Lambda=\Lambda(\rho)$ such that the $L^2-$norm of the solution is the  prescribed value.
 Set $\phi_\varepsilon(x):=\varepsilon V_{\xi_0}\(x-\xi_\varepsilon\over \varepsilon\)+\psi_\varepsilon(x)$. By   \eqref{restoref} and \eqref{pV}
 we immediately deduce that
 $$
 \(\int\limits_\Omega \phi^2_\varepsilon(x)dx\)^{\frac12}=\mathcal O\(\varepsilon^{\frac N2+1}\).
 $$
Then, taking into account   \eqref{la}, we get
 \begin{align}
\nonumber
\varepsilon^{-{4\over p-1}} \int\limits_{\Omega}u_{ \varepsilon}^2(x)dx&=  \varepsilon^{-{4\over p-1}}  \int\limits_{\Omega}\( U \({x-\xi_\varepsilon\over\varepsilon}\)+\phi_\varepsilon(x)\)^2dx
\\
\nonumber
  &=  \varepsilon^{-{4\over p-1}}\left[  \int\limits_{\Omega}  U ^2\({x-\xi_\varepsilon\over\varepsilon}\)dx+
  \int\limits_{\Omega}
  2U \({x-\xi_\varepsilon\over\varepsilon}\)\phi_\varepsilon(x) dx +
o(\varepsilon )  \right]
 \\
\nonumber
&=  \varepsilon^{-{4\over p-1}+N} \left[ \int\limits_{{\Omega-\xi_\varepsilon\over\varepsilon}} U^2 \(y\) dy+\mathcal O(\varepsilon)\right]
\\
\label{cru1}&= \varepsilon^{-{4\over p-1}+N}\left[   \sigma_0+\mathcal O(\varepsilon)\right]
=\rho\left[\Lambda (\rho)\sigma_0+\mathcal O(\varepsilon)\right]\
 \end{align}
 where the term $\mathcal O(\varepsilon)$   is uniform with respect to $\Lambda=\Lambda(\rho)$   when either $\rho\to+\infty$ or $\rho\to0.$

 Finally, we choose $\Lambda(\rho)$ as in \eqref{la}, when either $\rho\to+\infty$ or $\rho\to0,$ such that
 $$ \Lambda (\rho)\sigma_0+o(1) =1$$
and  by \eqref{cru1} we deduce that $u_\varepsilon$ has the prescribed
$L^2-$norm concluding the proof.
\end{proof}

In the critical case, namely when $p=\frac4N+1$ the situation
is more difficult and we can prove the following result.
 \begin{theorem}\label{main1crit}
 Let  $p=1+\frac4N$ and
 $\xi_0\in\partial\Omega$ be a non-degenerate critical point of the mean
curvature of the boundary $\partial\Omega$ such that $H(\xi_0)\not=0$. Suppose that
\begin{equation}\label{sigma1}
\mathfrak n:=    \int\limits_{\mathbb R^{N-1}}|y'|^2U^2 \(y',0\)dy'-(N-1)
\int\limits_{\mathbb R^N_+} U(y)W_1(y) dy\neq 0
\end{equation}
where $W_1$ is defined in \eqref{pW}.
Then,    there exists $\delta>0$
such that
if  either $\mathfrak n H(\xi_0)>0$ and
$\rho\in\(\sigma_0-\delta,\sigma_0\)$ or
$\mathfrak n H(\xi_0)<0$ and $\rho \in\(\sigma_0,\sigma_0+\delta\)$
(see \eqref{so}),
Problem \eqref{pnl}  with $\varepsilon:= \Lambda_\rho|\rho -\sigma_0|$ has a
solution $(\Lambda_\rho, u_\rho)$ such that $\Lambda_\rho \to {1\over
|H(\xi_0)\mathfrak n|}$  and $u_\rho$ concentrates at the
point $\xi_0$ as $\rho\to\sigma_0.$
 \end{theorem}
  \begin{proof}
In this case, let us fix
\begin{equation}\label{la2}
\varepsilon=\Lambda\delta\ \hbox{where}\ \delta:=|\rho -\sigma_0| \ \hbox{and}\ \Lambda=\Lambda(\delta)\in\left[
\frac1{2|H(\xi_0)\mathfrak n|},\frac 2{|H(\xi_0)\mathfrak n|}\right],
\end{equation}
where $\mathfrak n$ is defined in \eqref{sigma1}.
As in the proof of Theorem \ref{main1}
we have to choose the free parameter $\Lambda=\Lambda(\delta)$ such that   the $L^2-$norm of the solution is the  prescribed value $\rho$.
But, differently from the case $p\neq 1+\frac4N$, here, we need
 a more refined profile of the solution $u_\varepsilon$, namely we
 have to take into account the first order  $\varepsilon V_{\xi_0}\({x-\xi_\varepsilon\over\varepsilon}\)$ of the reimander term (see \eqref{refined}).
Indeed,  by \eqref{refined} and \eqref{restoref} we get
 \begin{equation}\label{cru2}
 \begin{aligned}
 \varepsilon^{-{4\over p-1}} \int\limits_{\Omega}u_{ \varepsilon}^2(x)dx&=   \int\limits_{{\Omega-\xi_\varepsilon\over\varepsilon}} U^2 \(y\) dy+2\varepsilon\int\limits_{{\Omega-\xi_\varepsilon\over\varepsilon}} U  \(y\) V_{\xi_0}(y)dy+\mathcal O\(\varepsilon^{\min\{2,p\}}\)
 \end{aligned}
 \end{equation}
 where the term $\mathcal O\left(\varepsilon^{\min\{2,p\}}\right)$   is uniform with respect to $\Lambda=\Lambda(\delta)$.\\
In order to compute the expansion of the right hand side of \eqref{cru2}
let us define
$$
B^+ :=\left\{x\in\mathbb R^N_+\ :\  |x|<\eta\right\}\ \hbox{and}\ \Sigma:=\left\{(x', x_N)\ :\ 0<x_N<\psi(x')\,\,: |x'|<\eta\right\}
$$
where the function $\psi$ given in \eqref{eq:defpsi}.
Rescaling $x=\varepsilon y+\xi_\varepsilon$ one sends
$B^+$ and $\Sigma$ to
$B^+_\varepsilon :=\left\{y\in\mathbb R^N_+\ :\ |y+\frac1\varepsilon{\xi_\varepsilon}|<\frac\eta\varepsilon\right\}$
and
\[
\Sigma_\varepsilon:=\left\{(y', y_N)\ :\ -\frac1 \varepsilon{{\xi_\varepsilon}_N}<y_N<\frac1\varepsilon\psi\(\varepsilon y'+\xi'_\varepsilon\)-\frac1\varepsilon{\xi_\varepsilon}_N
,\ |y'+\frac1\varepsilon\xi'_\varepsilon|<\frac \eta\varepsilon\right\}\subset
\frac{\Omega-\xi_\varepsilon}{\varepsilon}.
\]
Estimating the first term on the right hand side of \eqref{cru2}
and using  the decay properties of $U$ (see \eqref{groundstate})
one obtains
\begin{equation}\label{cru3}
\begin{aligned}
\int\limits_{ \Omega-\xi_\varepsilon\over\varepsilon}  U^2 \(y\) dy& = \int\limits_{B^+_\varepsilon}  U^2 \(y\) dy-
\int\limits_{\Sigma_\varepsilon} U^2 \(y\)  dy+\int\limits_{{ \Omega-\xi_\varepsilon\over\varepsilon} \setminus B^+_\varepsilon}  U^2 \(y\) dy\\
&= \int_{\mathbb R^N_+}  U^2 \(y\)  dy-
\int\limits_{\Sigma_\varepsilon} U^2 \(y\) dy+\mathcal O\left(\int\limits_{\mathbb R^N\setminus B^+_\varepsilon}  U^2 \(y\) dy\right) \\
&=\sigma_0 - \frac12 H(\xi_0)\varepsilon\int\limits_{\mathbb R^{N-1}}|y'|^2U^2 \(y',0\)dy'+o\(\varepsilon\).\end{aligned}
\end{equation}
Indeed,  \eqref{eq:defpsi}, standard computations together with the fact that
$\frac{\xi_\varepsilon}{\varepsilon}=o(1)$ ( as $\varepsilon\to0 $ (see \eqref{points} with $\xi_0=0$) show that
$$\begin{aligned}
\int\limits_{\Sigma_\varepsilon} U^2 \(y\) dy&=\int_{\left\{|y'+\frac1\varepsilon\xi'_\varepsilon|<\frac \eta\varepsilon\right\}}dy'\int_{-\frac1 \varepsilon{{\xi_\varepsilon}_N}}^{\frac1\varepsilon\psi\(\varepsilon y'+\xi'_\varepsilon\)-\frac1\varepsilon{\xi_\varepsilon}_N}U^2 \(y',y_N\) dy_N\\
&=\int_{\left\{|y'+\frac1\varepsilon\xi'_\varepsilon|<\frac \eta\varepsilon\right\}}\frac1\varepsilon\psi\(\varepsilon y'+\xi'_\varepsilon\)U^2 \(y',0\)dy'+
o\(\varepsilon\)\\
&=\frac12 H(\xi_0)
\varepsilon\int\limits_{\mathbb R^{N-1}}|y'|^2U^2 \(y',0\)dy'+o\(\varepsilon\).
\end{aligned}$$
With respect to the second term on the right hand side of \eqref{cru2}, we have
\begin{equation}\label{cru31}
\begin{aligned}
2\int\limits_{{\Omega-\xi_\varepsilon\over\varepsilon}} U  \(y\) V_{\xi_0}(y)dy
=&2\int\limits_{\mathbb R^N_+} U  \(y\) V_{\xi_0}(y)dy+o(1)
\\
=&  \sum\limits_{i=1}^{N-1}\kappa_i(\xi_0) \int\limits_{\mathbb R^N_+} U  \(y\) \left(
{\partial U\over\partial y_i}(y)y_iy_N+W_i(y)\right)dy+o(1)
\\
=& -\frac 1{2 }H(\xi_0)\int\limits_{\mathbb R^{N-1}}
 U^2(y',0) |y'|^2 dy'
\\
&+\sum\limits_{i=1}^{N-1}\kappa_i(\xi_0)
\int\limits_{\mathbb R^N_+} U(y)W_1(y) dy +o(1),
 \end{aligned}
 \end{equation}
since
$$\begin{aligned}
 \int\limits_{\mathbb R^N_+} U  \(y\)
{\partial U\over\partial y_i}(y)y_iy_N dy&=\int\limits_{\mathbb R^N_+} U  \(y\)
{\partial U\over\partial y_N}(y)y_i^2dy=\frac 12\int\limits_{\mathbb R^N_+}
{\partial U^2\over\partial y_N}(y)y_i^2dy\\
& =\frac 12\int\limits_{\mathbb R^N_+}
{\partial \over\partial y_N}\(U^2(y) y_i^2\)dy =-\frac 12\int\limits_{\mathbb R^{N-1}}
 U^2(y',0) y_i^2 dy'\\
& =-\frac 1{2(N-1)}\int\limits_{\mathbb R^{N-1}}
 U^2(y',0) |y'|^2 dy'.
 \end{aligned}
$$
Combining \eqref{cru2}, \eqref{cru3} and \eqref{cru31} together with  \eqref{sigma1} and the choice of $\varepsilon$ in \eqref{la2}, we get
\begin{equation}\label{cru4}
\begin{aligned}
\varepsilon^{-{4\over p-1}} \int\limits_{\Omega}u_{ \varepsilon}^2(x)dx&=  \sigma_0- H(\xi_0)\mathfrak n\varepsilon+o\(\varepsilon\)
=\sigma_0- H(\xi_0)\mathfrak n\Lambda\delta+o\(\delta\)\\
 &=\rho \pm\delta - H(\xi_0)\mathfrak n\Lambda\delta+o\(\delta\)
\end{aligned}
\end{equation}
where the term $o(\cdot)$   is uniform with respect to $\Lambda=\Lambda(\delta)$.   \\
Finally, it is clear that it is possible to choose $\Lambda(\delta)$ as in \eqref{la}, when   $\delta\to0,$ such that
 $$
 1 - H(\xi_0)\mathfrak n\Lambda(\delta)+o\(1\) =0\ \hbox{or}\   -1 - H(\xi_0)\mathfrak n\Lambda(\delta)+o\(1\) =0
 $$
 (in particular $H(\xi_0)\mathfrak n>0$ in the first case and $H(\xi_0)\mathfrak n<0$ in the second case)
and  by \eqref{cru4} we deduce that $u_\varepsilon$ has the prescribed $L^2-$norm. That concludes the proof.
\end{proof}

 \begin{remark}
 We point out that (i) and (ii) of Theorem \ref{main1} hold true when $\xi_0$ is a $C^1-$stable critical point of the mean curvature according the definition given by Li in \cite{yyl}.
 The non-degeneracy assumption is used in proving (iii) of Theorem \ref{main1}, since  it ensures the estimate \eqref{points} which  turns to be crucial in the second order expansion of the $L^2-$norm of the solution.
 \\
It is useful to recall that Micheletti and Pistoia in \cite{mp} proved that for
generic  domains $\Omega$ the mean curvature of the boundary is a Morse function, i.e. all its critical points are non-degenerate.
 \end{remark}

\begin{remark}\label{rm:morsebordo}
We point out that  if $\xi_0$ is a non-degenerate critical  point of the mean curvature  of the boundary whose index Morse  is $m(\xi_0)$
then by Theorem 4.6  in \cite{bash} we deduce that
the solution   concentrating at a $\xi_0$ is non-degenerate and has Morse index $1+m(\xi_0)$.
In particular, the solution  concentrating at a non-degenerate minimum point of the mean curvature  of the boundary is non-degenerate and has Morse index 1.
\end{remark}

\subsection{Interior concentration}\label{sub:int}

 In this subsection we will find normalized solutions concentrating at an
 interior point of the bounded domain $\Omega$. Our analysis is
 based on well known results proved  by  Gui, Ni and Wei  in \cite{gw,nw,w1,w2,w3},  concerning the existence of solutions to
 the  following Dirichlet and Neumann problem
  \begin{equation}\label{pdn}
  \begin{cases}
  -\varepsilon^2\Delta u+  u = u^p &\hbox{in}\ \Omega,
  \\
u>0 &\hbox{in}\ \Omega,
\\
u=0  \;\hbox{or}\; \partial_\nu u=0 &\hbox{on}\ \partial\Omega
\end{cases}
\end{equation}
 as $\varepsilon$ is small enough.\\
In order to summarize the afore mentioned results, let us first state the
following proposition (see Lemma 4.3 and 4.4 in \cite{nw} and Section 3 in
\cite{w4}).
 \begin{proposition}\label{visco}
  Let $U_{\varepsilon,\xi}(x):=U\(\frac{x-\xi }\varepsilon\)$ for $x,\xi\in\Omega $ and  let $\varphi_{\varepsilon,\xi}$ be the solution
 to the problem
\begin{equation}\label{prob:fi}
 \begin{cases}
 -\varepsilon^2\Delta \varphi_{\varepsilon,\xi}+  \varphi_{\varepsilon,\xi} =0  &\hbox{in}\ \Omega,\
\\
\varphi_{\varepsilon,\xi}=U_{\varepsilon,\xi}\ \hbox{or}\ \partial_\nu
\varphi_{\varepsilon,\xi}=\partial_\nu U_{\varepsilon,\xi} & \hbox{on}\ \partial\Omega.
\end{cases}
\end{equation}
Set
$$
\psi_{\varepsilon }(\xi):=-\varepsilon\ln\(\varphi_{\varepsilon,\xi}(\xi)\)\ \hbox{
in case of  Dirichlet boundary conditions}
$$
 or
 $$
\psi_{\varepsilon }(\xi):=-\varepsilon\ln\(-\varphi_{\varepsilon,\xi}(\xi)\)\ \hbox{
in case of  Neumann boundary conditions.}$$
Then
$$\lim\limits_{\varepsilon\to0}\psi_{\varepsilon }(\xi)=2d_{\partial\Omega}(\xi)\ \hbox{uniformly in}\ \Omega.$$
where
$d_{\partial\Omega}(\xi):= {\rm dist}(\xi,\partial\Omega).$
\end{proposition}
Now, we can state the existence result  (see Lemma 2.1 of \cite{gp})
 \begin{theorem}\label{www} Let  $\xi_0\in\Omega$ be the maximum point of
 the distance function from the boundary $\partial\Omega.$
There exists $\varepsilon_0>0$ such that for any $\varepsilon\in(0,\varepsilon _0)$
there exists a solution $u_\varepsilon$ to \eqref{pdn}
 which concentrates at the point $\xi_0$   as $\varepsilon\to0.$ More precisely,
\begin{equation}\label{ue}u_\varepsilon(x)= U\(\frac{x-\xi_\varepsilon}\varepsilon\)-
\varphi_{\varepsilon,\xi_\varepsilon}(x)+\phi_{\varepsilon,\xi_\varepsilon} (x)
\end{equation}
 where
\begin{equation}\label{pointse}
\xi_\varepsilon\to \xi_0 \ \hbox{as}\
\varepsilon\to0\ \hbox{with}\ d_{\partial\Omega}
(\xi_0)=\max\limits_{\xi\in\Omega}d_{\partial\Omega}(\xi)
\end{equation}
and
\begin{equation}\label{restoe} \|\phi_{\varepsilon,\xi_\varepsilon}\|
_{H^1_\varepsilon(\Omega)}:=\(\int\limits_{\Omega}\(\varepsilon^2|\nabla
\phi_{\varepsilon,\xi_\varepsilon}|^2+ \phi_{\varepsilon,\xi_\varepsilon}^2\)dx\)^{1/2}=
 \mathcal O\(\varepsilon^{N\over2}|\varphi_{\varepsilon,\xi_\varepsilon}(\xi_\varepsilon)|^{\min\{1,p/2\}}\).
\end{equation}
 \end{theorem}
From the above result we obtain the asymptotic behaviour of
$\phi_{\varepsilon,\xi_\varepsilon}$ in dependence on $\varphi_{\varepsilon,\xi}$, whereas
the following Lemma gives an analogous first information on  $\varphi_{\varepsilon,\xi}$; note that differently from the case of boundary concentration,
here $\varphi_{\varepsilon,\xi}$ decays exponentially as $\epsilon\to 0$.
  \begin{lemma}\label{fi}
 For any $\delta>0$ there exist $\varepsilon_0>0$, $\eta>0$ and $C>0$ such that for any $\varepsilon\in(0,\varepsilon_0)$ and $\xi\in\Omega$ such that $d_{\partial\Omega}(\xi)\ge\delta$ it holds true
 $$\| \varphi_{\varepsilon,\xi}\|_{L^\infty(\Omega)}\le C e^{-{d_{\partial\Omega}(\xi)\over\varepsilon}}.
 $$
 \end{lemma}
 \begin{proof}
Arguing as in Section 7 of \cite{w4} and taking into account
Remark \ref{wei_correction1}, we immediately deduce
\begin{equation}\label{fi2}\begin{aligned}|\varphi_{\varepsilon,\xi}(x)|&\le C\int\limits_{\partial\Omega}e^{-{|z-\xi|+|z-x|\over\varepsilon}}|z-\xi|^{-{N-1\over2}}
 |z-x|^{-{N-1\over2}}\left\langle{z-x\over|z-x|},\nu\right\rangle dz\\ &\le  Ce^{-{d_{\partial\Omega}(\xi)\over\varepsilon}}\( d_{\partial\Omega}(\xi)\)^{-{N-1\over2}}\int\limits_{\partial\Omega}
 |z-x|^{-{N-1\over2}} dz.
 \end{aligned}
 \end{equation}
Then, in order to conclude the proof it is enough to show that
\begin{equation}\label{fi3}
\left\|\int\limits_{\partial\Omega}
 |z-x|^{-{N-1\over2}} dz\right\|_{L^\infty(\Omega)}\le C.
 \end{equation}
Let $\delta>0$ be fixed and small enough so that for any $x\in\Omega$ such that   $d_{\partial\Omega}(x)\le \delta$ there exists a unique $\pi_x\in\partial\Omega$ such that $|\pi_x-x|=d_{\partial\Omega}(x).$ Now it is clear that
 $$\int\limits_{\partial\Omega}
 |z-x|^{-{N-1\over2}} dz\le  \delta^{-{N-1\over2}}|\partial\Omega|\ \hbox{for any $x\in\Omega$ such that   $d_{\partial\Omega}(x)\ge \delta$}.$$
Let us consider the case  $d_{\partial\Omega}(x)\le \delta.$ By the choice of $\delta$,  we can write $x=\pi_x+d_{\partial\Omega}(x)\nu_{\pi_x},$ where $\nu_{\pi_x}$ denotes the inward normal at the boundary point $\pi_x.$ We remark that,
since $\partial\Omega$ is $C^2,$  there exists a constant $L$ such that
\[
 |\langle z-w,\nu_z\rangle|\le L |z-w|^2\ \hbox{for any}\ z,w\in \partial\Omega,
\]
and this implies
 $$
 \begin{aligned}|z-x|^2&=|z-\pi_x-d_{\partial\Omega}(x)\nu_{\pi_x}|^2=|z-\pi_x|^2+d^2_{\partial\Omega}(x)-2d_{\partial\Omega}(x)\langle z-\pi_x,\nu_{\pi_x}\rangle\\
 &
 \ge |z-\pi_x|^2\(1-2L d_{\partial\Omega}(x)\)+d^2_{\partial\Omega}(x)
 \geq |z-\pi_x|^2\(1-2L \delta\)
 \ge
 \frac12 |z-\pi_x|^2
 \end{aligned}
 $$
choosing $\delta$ so that $1-2L\delta>1/2.$
 Therefore,  it is immediate to check that there exists $C>0$ such that
 $$\int\limits_{\partial\Omega}
 |z-x|^{-{N-1\over2}} dz\le 2^{-{N-1\over2}}\int\limits_{\partial\Omega}
 |z-\pi_x|^{-{N-1\over2}} dz\le C,\quad \hbox{ $\forall x\in\Omega :d_{\partial\Omega}(x)\le \delta$}.
 $$
 That concludes the proof of \eqref{fi3}.
 \end{proof}

In the above lemma we have used the following representation formula
for $\varphi_{\varepsilon,\xi}(x)$.
\begin{remark}\label{wei_correction1}
 Let $G_\varepsilon (\cdot,P),$ $P\in\Omega,$ the Green's function of $-\varepsilon^2\Delta+1$ in $\Omega$ with Dirichlet or Neumann boundary condition. Let $\tilde G_\varepsilon(\cdot, P ),$  the Green's function of $- \Delta+1$ in the scaled domain $\Omega_{\varepsilon }:=\Omega/\varepsilon$ with Dirichlet or Neumann boundary condition.
We claim that
$$G_\varepsilon (x,P)={1\over\varepsilon^n}\tilde G_\varepsilon\(x/\varepsilon,P\)$$
 Indeed by changing   variable $\varepsilon y=x$ we get
$$ \int\limits_\Omega \( -\varepsilon^2\Delta_x G_\varepsilon (x,P)+G_\varepsilon (x,P)\)dx=\int\limits_{\Omega/\varepsilon} \( - \Delta_y \tilde G_\varepsilon (y,P)+\tilde G_\varepsilon (y,P)\)dy=1.$$
Therefore, formulas (7.4) in \cite{w4} and (9.2) in \cite{w1} have to be corrected as follows
\[
\varphi_{\varepsilon,P}(x)=\pm\({c_N}+o(1)\)
\int\limits_{\partial\Omega}e^{-{|z-P|+|z-x|\over\varepsilon}} |z-P|^{-{N-1\over2}}|z-x|^{-{N-1\over2}}{{\langle z-x,\nu\rangle}\over|z-x|}dz,
\]
where the sign + is taken in the Dirichlet case and the sign -  in the Neumann case.
\end{remark}

We are now in the position to tackle both the Dirichlet and Neumann problems with prescribed $L^2-$norm
  \begin{equation}\label{pdn2}
  \begin{cases}
  -\varepsilon^2\Delta u+   u = u^p & \hbox{in}\ \Omega,
 \\ u>0 & \hbox{in}\ \Omega,
 \\
  u=0\ \hbox{or} \ \partial_\nu u=0 & \hbox{on}\ \partial\Omega,
 \\
 \varepsilon^{-{4\over p-1}}\int\limits_\Omega u^2 =\rho.
 \end{cases} \end{equation}
 \begin{theorem}\label{main2}
 Let $\xi_0\in \Omega$ be the maximum point of the distance function from the $\partial\Omega.$
 \begin{itemize}
 \item[(i)] If $p<\frac 4N+1$ there exists $R>0$ such that for any $\rho>R$
Problem \eqref{pdn2}  has a solution $(\Lambda_\rho, u_\rho)$
for $\varepsilon:=\(\Lambda_\rho\rho\)^{(p-1)\over
(p-1)N-4}$  with $\Lambda_\rho
\to {1\over {2\sigma_0}}$ and $u_\rho$   concentrating at the point $
\xi_0$ as $\rho\to\infty$.
 \item[(ii)]  If $p>\frac 4N+1$ there exists $r>0$ such that for any $\rho<r$
Problem \eqref{pdn2} has a solution $(\Lambda_\rho, u_\rho)$ for
 $\varepsilon:=\(\Lambda_\rho\rho\)^{(p-1)\over
(p-1)N-4}$  with $\Lambda_\rho
\to {1\over {2\sigma_0}}$ and $u_\rho$    concentrating at the point $
\xi_0$ as $\rho\to0$.
  \end{itemize}
 \end{theorem}
 \begin{proof}
We want to reduce the existence of solutions to problem \eqref{pdn2} with
variable but prescribed $L^2-$norm to the existence of solutions to problem
\eqref{pdn} where the parameter $\varepsilon$ is small. Let us choose
 \begin{equation}\label{la22}
 \varepsilon^{-{4\over p-1}+   N }=\Lambda \rho\ \hbox{with}\ \Lambda=\Lambda(\rho)\in\left[\frac{1}{4  {\sigma_0}} ,\frac1{ {\sigma_0}}\right]
 \end{equation}
 where  $\sigma_0$ is defined in \eqref{so}.
 It is clear that $\varepsilon\to0$ if and only if either
  $ p<\frac 4N +1$ and $\rho\to\infty$ or
  $  p>\frac 4N +1$ and $\rho\to0.$

 By Theorem \ref{www} we deduce that for any $\Lambda$ as in \eqref{la22}, there exists either $R>0$ or $r>0$ such that for any $\rho>R$ or $\rho<r$ problem \eqref{pn} has a solution $u_{\varepsilon}$ as in \eqref{ue} such that $\varepsilon$ satisfies \eqref{la22}.
 Now, we have to choose the free parameter $\Lambda=\Lambda(\rho)$ such that the $L^2-$norm of the solution is $\rho.$ Lemma \ref{fi}
and \eqref{restoe} yield
\begin{equation*}
\begin{aligned}\varepsilon^{-{4\over p-1}}
 \int\limits_\Omega u_\varepsilon^2(x)dx
 =&
 \varepsilon^{-{4\over p-1}}\int\limits_\Omega \(U\(\frac{x-\xi_\varepsilon}\varepsilon\)-\varphi_{\varepsilon,\xi_\varepsilon}(x)+\phi_{\varepsilon,\xi_\varepsilon} (x)\)^2dx
 \\
 =&
 \varepsilon^{-{4\over p-1}}
 \left[
 \int\limits_\Omega U^2\(\frac{x-\xi_\varepsilon}\varepsilon\)dx-2\int\limits_\Omega\varphi_{\varepsilon,\xi_\varepsilon}(x)U\(\frac{x-\xi_\varepsilon}\varepsilon\)dx\right.
 \\ &\left.
 +\int\limits_\Omega \varphi^2_{\varepsilon,\xi_\varepsilon}(x)dx \right.
 \\ &\left.
 +2\int\limits_\Omega \(U\(\frac{x-\xi_\varepsilon}\varepsilon\)-\varphi_{\varepsilon,\xi_\varepsilon}(x)\)\phi_{\varepsilon,\xi_\varepsilon} (x) dx+
 \int\limits_\Omega \phi^2_{\varepsilon,\xi_\varepsilon} (x) dx\right]
 \\
 =&
 \varepsilon^{-{4\over p-1}+N}\left[
 \ \int\limits_{\mathbb R^N} U^2\(y\)dy+o(1)
 \right]
 \\
 =&\rho \left[2 \Lambda(\rho)\sigma_0+o(1)\right]
 \end{aligned}\end{equation*}
 where the last equality comes from \eqref{la22}.
 Finally, it is clear that it is possible to choose $\Lambda(\rho)$ as in \eqref{la22}, when either $\rho\to+\infty$ or $\rho\to0,$ such that
 $$
2 \Lambda (\rho)\sigma_0+o(1) =1
 $$
which is immediately satisfied for
$\Lambda (\rho)=\frac1{2\sigma_{0}}+o(1)$.
Then  $ u_\varepsilon $ has the prescribed $L^2-$norm and  the proof is completed.\\
\end{proof}

 \begin{remark} We point out that  the existence result Theorem \ref{www} holds true  when $\xi_0$ is a  stable critical point of the distance function from the boundary as pointed out by Grossi and Pistoia in \cite{gp}.
 Therefore, also
 (i) and (ii) of Theorem \ref{main2} holds true in this more general situation.
 \end{remark}

 \subsubsection{The critical case}
Let us consider
the critical case $p=\frac 4 N+1$. This is in general quite difficult to deal with.
We will prove the following result.
 \begin{theorem}\label{main2critico}
 Let $p= 1+ \frac 4N$, $\sigma_0$ be defined as in \eqref{so}, and $\xi_0\in \Omega$ be the maximum point of the distance function from the $\partial\Omega$.
 \begin{itemize}
 \item[(i)] In the case of Dirichlet boundary conditions, there exists $0<r<2\sigma_0$ such that for any $r<\rho<2\sigma_0$ Problem \eqref{pdn2} has a solution $(\eps_\rho, u_\rho)$ such that
  $\eps_\rho\to0$ and $u_\rho$   concentrates at the point $\xi_0$ as $\rho\to2\sigma_0^-$.
 \item[(ii)]  In the case of Neumann boundary conditions, there exists $R>2\sigma_0$ such that for any $2\sigma_0<\rho<R$ Problem \eqref{pdn2} has a solution $(\eps_\rho, u_\rho)$ such that
  $\eps_\rho\to0$ and $u_\rho$   concentrates at the point $\xi_0$ as $\rho\to2\sigma_0^+$.
  \end{itemize}
 \end{theorem}
Notice that in this result we only know that $\eps_\rho = o(1)$ as $\rho\to2\sigma_0$, and we can provide the exact asymptotics only in dimension $N=1$, see Remark \ref{rem:DirN=1} ahead.
\\
In the proof of the above result we will need a deeper comprehension
on the asymptotical behavior of $\varphi_{\varepsilon,\xi_\varepsilon}$.
Following   \cite{nw,w1,w2,w3},
 set
$$V_{\varepsilon,\xi}(y):={\varphi_{\varepsilon,\xi}(\varepsilon y+\xi)\over\varphi_{\varepsilon,\xi}(\xi)},\ y\in\Omega_{\varepsilon,\xi}:={\Omega-\xi\over\varepsilon}.$$
Then for any sequence $\varepsilon_n\to0$ there exists a subsequence ${\varepsilon_n}_k$ such that
$$V_{{\varepsilon_n}_k,\xi}\to V_\xi\ \hbox{uniformly on compact sets of $\mathbb R^N$},$$
where
\begin{equation}\label{v1}
V_\xi(y)=\int\limits_{\partial\Omega} e^{\langle{\zeta-\xi\over|\zeta-\xi|},y\rangle}d\mu_{\xi}(\zeta)\end{equation}
where $d\mu_{\xi}$ is a bounded Borel measure on $\partial \Omega$ with $\int\limits_{\partial\Omega}d\mu_{\xi}(\zeta)=1$
and supp$\(d\mu_{\xi}\)\subset\{\zeta\in\partial\Omega\ :\ |\zeta-\xi|=d_{\partial\Omega}(\xi)\}.$
Moreover for any $\eta>0$ it holds true
$$\sup\limits_{y\in\Omega_{{\varepsilon_n}_k,\xi}}e^{-(1+\eta)|y|}\left|V_{{\varepsilon_n}_k,\xi}(y)- V_\xi(y)\right|\to 0\ \hbox{as}\ {\varepsilon_n}_k\to0.$$

\begin{lemma}\label{lemma_cri1}
Let $p=1+\frac4N$ and define
\begin{equation}\label{eq:deftheta}
\Theta_\varepsilon:= \int\limits_{\frac{\Omega-\xi_\varepsilon}\varepsilon}\varphi_{\varepsilon,\xi_\varepsilon}(\varepsilon y+\xi_\varepsilon)U\(y\)dy.
\end{equation}
Then, $\Theta_\varepsilon=o(1)$ as $\varepsilon\to0$ and it holds
\begin{equation}\label{claim}
\begin{aligned}
\varphi_{\varepsilon,\xi_\varepsilon} (\xi_\varepsilon)=o\(  \Theta_\varepsilon \).
\end{aligned}
\end{equation}
\end{lemma}
\begin{proof}
First, applying Lemma \ref{fi} one gets that  $\Theta_\eps = o(1)$.
For any $R>0$
$$
\frac{\Theta_\varepsilon}{\varphi_{\varepsilon,\xi_\varepsilon}( \xi_\varepsilon)}
=
\int\limits_{\frac{\Omega-\xi_\varepsilon}\varepsilon}V_{\varepsilon,\xi_\varepsilon}(y)U\(y\)dy
\ge  \int\limits_{B( \xi_\varepsilon,R)}V_{\varepsilon,\xi_\varepsilon}(y)U\(y\)dy
$$
and by \eqref{v1} we get
$$
\liminf_{\varepsilon\to0} \frac{\Theta_\varepsilon}{\varphi_{\varepsilon,\xi_\varepsilon}( \xi_\varepsilon)}
  \ge  \int\limits
  _{B( \xi ,R)}U(y)dy\int\limits_{\partial\Omega} e^{\langle{\zeta-\xi\over|\zeta-\xi|},y\rangle}d\mu_{\xi}(\zeta)
   $$
  and letting $R\to+\infty$ we immediately get  \eqref{claim} since
  $$\lim_{\varepsilon\to0} \frac1{\varphi_{\varepsilon,\xi_\varepsilon}( \xi_\varepsilon)}\int\limits
  _{\frac{\Omega-\xi_\varepsilon}\varepsilon}\varphi_{\varepsilon,\xi_\varepsilon}(\varepsilon y+\xi_\varepsilon)U\(y\)dy=+\infty,$$
  because the function
  $$y\to U(y) \int\limits_{\partial\Omega} e^{\langle{\zeta-\xi\over|\zeta-\xi|},y\rangle}d\mu_{\xi}(\zeta)\not\in L^1(\mathbb R).$$
This concludes the proof.
\end{proof}
 As a consequence of Lemma \ref{lemma_cri1} we will get that the leading term of the $L^2-$norm of the solution is
$$
\varepsilon^{-{4\over p-1}}\int\limits_\Omega u_\varepsilon^2(x)dx
\sim 2\sigma_0 -2 \Theta_\varepsilon.
$$
In general  it is difficult to find the exact rate  of $\Theta_\varepsilon$ in terms of $\varepsilon $ and this is why we cannot choose the parameter $\varepsilon$ in terms of the prescribed norm $\rho$  as in Theorem \ref{main1crit} and  Theorem \ref{main3} - (iii).\\
We are now in the position to give the proof of Theorem \ref{main2critico}.
\begin{proof}[Proof of Theorem \ref{main2critico}]
Taking into account \eqref{ue} and \eqref{eq:deftheta} we get
\begin{equation}\label{eq:contol2}
\begin{aligned}
\varepsilon^{-{4\over p-1}}
 \int\limits_\Omega u_\varepsilon^2(x)dx
 =
 &  \int\limits_{\mathbb R^N} U^2\(y\)dy-
 \int\limits_{\mathbb R^N\setminus \frac{\Omega-\xi_\varepsilon}\varepsilon} U^2\(y\)dy  -2\Theta_\varepsilon+\varepsilon^{-N}\int\limits_{\Omega} \varphi^2_{\varepsilon,\xi_\varepsilon}(x)dx
 \\ &
 +2 \varepsilon^{-N}\int\limits_{\Omega}  U\(\frac{x-\xi_\varepsilon}\varepsilon\) \phi_{\varepsilon,\xi_\varepsilon} (x)   -2 \varepsilon^{-N}\int\limits_{\Omega}  \varphi_{\varepsilon,\xi_\varepsilon}(x)\phi_{\varepsilon,\xi_\varepsilon} (x) dx
  \\ &+
 \varepsilon^{-N}\int\limits_{\Omega}\phi^2_{\varepsilon,\xi_\varepsilon} (x) dx.  \\
\end{aligned}\end{equation}
Let us estimate all the right-hand side terms of this formula.
First of all,  taking into account the size of the error \eqref{restoe}, we get
\begin{equation}\label{eq:resto}
\varepsilon^{-N}\int\limits_{\Omega}\phi^2_{\varepsilon,\xi_\varepsilon} (x) dx=\mathcal O\(|\varphi_{\varepsilon,\xi_\varepsilon} (\xi_\varepsilon)|^{\min\{2,p\}}\)=o  \( |\varphi_{\varepsilon,\xi_\varepsilon} (\xi_\varepsilon)|\).
\end{equation}
In addition, recalling that $U$ is the solution of \eqref{pblim}, we obtain that
the function $U_\varepsilon(x):= U\(\frac{x-\xi_\varepsilon}\varepsilon\)$
satisfies
$$
\int\limits_{\mathbb R^N\setminus \Omega}\(\varepsilon^2|\nabla U_\varepsilon|^2+U_\varepsilon^2\)dx=
\int\limits_{\mathbb R^N\setminus \Omega} U_\varepsilon^{p+1}dx+\varepsilon^2\int\limits_{\partial\Omega} \partial_\nu U_\varepsilon U_\varepsilon dz
$$
so that
 $$
\begin{aligned}
\int\limits_{\mathbb R^N\setminus {\Omega-\xi_\varepsilon\over\varepsilon}} U^2\(y\)dy& =\varepsilon^{-N}\int\limits
  _{\mathbb R^N\setminus \Omega} U^2\(\frac{x-\xi_\varepsilon}\varepsilon\)dx\le \varepsilon^{-N}\int\limits
  _{\mathbb R^N\setminus \Omega} U^{p+1}\(\frac{x-\xi_\varepsilon}\varepsilon\)dx+\\
 & \qquad+\varepsilon^{2-N}\int\limits_{\partial\Omega}
  U \(\frac{z-\xi_\varepsilon}\varepsilon\)\frac 1\varepsilon  U' \(\frac{z-\xi_\varepsilon}\varepsilon\) {{\langle z-\xi_\varepsilon,\nu\rangle}\over|z-\xi_\varepsilon|}dz\\
  &=\mathcal O\( |\varphi_{\varepsilon,\xi_\varepsilon} (\xi_\varepsilon)|\).  \end{aligned}$$
Let us explain why the last equality holds.  From \eqref{pointse}
we deduce that for $\epsilon$ sufficiently small $B(\xi_\varepsilon,d_{\partial\Omega}(\xi_\varepsilon)\subset \Omega$;
then \eqref{groundstate} yields
$$\begin{aligned}\varepsilon^{-N}\int\limits
  _{\mathbb R^N\setminus \Omega} U^{p+1}\(\frac{x-\xi_\varepsilon}\varepsilon\)dx&\le
\varepsilon^{-N}\int\limits
  _{\mathbb R^N\setminus B(\xi_\varepsilon,d_{\partial\Omega}(\xi_\varepsilon)} U^{p+1}\(\frac{x-\xi_\varepsilon}\varepsilon\)dx\\ &=\mathcal O\(
    \varepsilon^{(p+1)\frac{N-1}{2}-N}e^{-(p+1){d_{\partial\Omega}(\xi_\varepsilon)\over\varepsilon}}\)\\
    &=  o\( |\varphi_{\varepsilon,\xi_\varepsilon} (\xi_\varepsilon)|\).\end{aligned}$$
Moreover, from \eqref{pointse} we get that $\frac{|z-\xi_{\epsilon}|}{\epsilon}\to+\infty$ for every $z\in \partial \Omega$, so that from
\eqref{groundstate} and using  the expression of $\varphi_{\varepsilon,\xi_\varepsilon} $ given in Remark \ref{wei_correction1} we get
$$
\begin{aligned}
&\varepsilon^{2-N}\left| \,\int\limits_{\partial\Omega}
  U \(\frac{z-\xi_\varepsilon}\varepsilon\)\frac 1\varepsilon  U' \(\frac{z-\xi_\varepsilon}\varepsilon\) {{\langle z-\xi_\varepsilon,\nu\rangle}\over|z-\xi_\varepsilon|}dz\right|
\\ &
  = \varepsilon^{1-N}\left| \,  \int\limits_{\partial\Omega}  e^{-{2|z-\xi_\varepsilon|\over\varepsilon}} \left|{z-\xi_\varepsilon\over \varepsilon}\right|^{-(N-1)} {{\langle z-\xi_\varepsilon,\nu\rangle}\over|z-\xi_\varepsilon|}(\mathfrak c+o(1))dz\right|
\\ &
= \mathcal O\( |\varphi_{\varepsilon,\xi_\varepsilon} (\xi_\varepsilon)|\).\end{aligned}
$$
Using these asymptotical information and taking into account
\eqref{eq:resto}, \eqref{eq:contol2} becomes
\begin{equation}\label{eq:cru3}
\begin{aligned}
\varepsilon^{-{4\over p-1}}
 \int\limits_\Omega u_\varepsilon^2(x)dx
 =
 & 2\sigma_{0}
 -2\Theta_\varepsilon+\varepsilon^{-N}\int\limits_{\Omega} \varphi^2_{\varepsilon,\xi_\varepsilon}(x)dx
-2\varepsilon^{-N}\int\limits_{\Omega}  \varphi_{\varepsilon,\xi_\varepsilon}(x)\phi_{\varepsilon,\xi_\varepsilon} (x) dx
\\ &
+2 \varepsilon^{-N}\int\limits_{\Omega}  U\(\frac{x-\xi_\varepsilon}\varepsilon\) \phi_{\varepsilon,\xi_\varepsilon} (x)
+\mathcal O\( |\varphi_{\varepsilon,\xi_\varepsilon} (\xi_\varepsilon)|\).
\end{aligned}
\end{equation}
Let us now study the last three integral terms on the right hand side.
One has
\begin{equation}\label{eq:fine1}
\begin{aligned}
\varepsilon^{-N}\int\limits_{\Omega}  U\(\frac{x-\xi_\varepsilon}\varepsilon\) \phi_{\varepsilon,\xi_\varepsilon} (x) dx
&=\mathcal O\( \(\varepsilon^{-N}\int\limits_{\Omega}\phi^2_{\varepsilon,\xi_\varepsilon} (x) dx\)^{1/2}\)\\ &=\mathcal O\(|\varphi_{\varepsilon,\xi_\varepsilon} (\xi_\varepsilon)|^{\min\{1,p/2\}}\)=\mathcal O  \( |\varphi_{\varepsilon,\xi_\varepsilon} (\xi_\varepsilon)|\)
\end{aligned}
\end{equation}
if $p\ge2,$ i.e. in low dimension $N=1,2,3,4$.  In higher dimension  the estimate   is quite delicate and we need to use some careful estimates of the error term $ \phi_{\varepsilon,\xi_\varepsilon} $ proved by Ni-Wei in   \cite{nw} (see  page 752) in the Dirichlet case and by Wei in \cite{w1} (see   page 871) in the Neumann case.
 More precisely,   it is proved that if $\mu<1$ is close enough to 1 and fixed then
\begin{equation}\label{cruciale}\left| \phi_{\varepsilon,\xi_\varepsilon} (\varepsilon y+\xi_\varepsilon)\over \varphi_{\varepsilon,\xi_\varepsilon}(\xi_\varepsilon)\right|\le C e^{\mu |y|} \ \hbox{for any}\ y\in {\Omega-\xi_\varepsilon\over\varepsilon} \end{equation}
where the constant $C$ does not depend on $\varepsilon$ when $\varepsilon$ is small enough. Therefore, from \eqref{groundstate} and \eqref{cruciale} it
follows
$$
\begin{aligned}
\varepsilon^{-N}\int\limits_{\Omega}  U\(\frac{x-\xi_\varepsilon}\varepsilon\)
\phi_{\varepsilon,\xi_\varepsilon} (x) dx
&= \int\limits_{\Omega-\xi_\varepsilon\over\varepsilon}  U\(y\)
\phi_{\varepsilon,\xi_\varepsilon} (\varepsilon y+\xi_\varepsilon)dy\\
&= \varphi_{\varepsilon,\xi_\varepsilon}(\xi_\varepsilon) \int\limits_{\Omega-
\xi_\varepsilon\over\varepsilon}  U\(y\) {\phi_{\varepsilon,\xi_\varepsilon}
(\varepsilon y+\xi_\varepsilon)\over \varphi_{\varepsilon,\xi_\varepsilon}
(\xi_\varepsilon)} dy
\\&
=\mathcal O\(|\varphi_{\varepsilon,\xi_\varepsilon}(\xi_\varepsilon)| \).
\end{aligned}
$$
Using these information in \eqref{eq:cru3}, we obtain
\begin{equation}\label{cru12}
\begin{aligned}
\varepsilon^{-{4\over p-1}}
 \int\limits_\Omega u_\varepsilon^2(x)dx
 =
 & 2\sigma_{0}-2\Theta_\varepsilon+\varepsilon^{-N}\int\limits_{\Omega} \varphi^2_{\varepsilon,\xi_\varepsilon}(x)dx
-2\varepsilon^{-N}\int\limits_{\Omega}  \varphi_{\varepsilon,\xi_\varepsilon}(x)\phi_{\varepsilon,\xi_\varepsilon} (x) dx
  \\ &+ \mathcal O \( |\varphi_{\varepsilon,\xi_\varepsilon} (\xi_\varepsilon)|\).
\end{aligned}\end{equation}

The study of the last two terms is quite delicate. First of all,
taking into account that $\varphi_{\varepsilon,\xi_\varepsilon}$
solves \eqref{prob:fi}, we get
$$\varepsilon^2\int\limits_\Omega|\nabla \varphi|^2dx+\int\limits_\Omega \varphi ^2dx=\varepsilon^2\int\limits_{\partial\Omega} \partial_\nu \varphi(z)\varphi (z)dz,$$
which implies
$$
\begin{aligned}\int\limits_\Omega \varphi ^2dx&\le \varepsilon^{2}\int\limits_{\partial\Omega} \partial_\nu \varphi(z)
\varphi (z)dz.
\end{aligned}
$$
Now, let us remind that on the boundary $\partial\Omega$ we have in the Dirichlet case
$$\varphi(z)=U\(z-\xi_\varepsilon\over\varepsilon\) $$
and  by Lemma 8.1 in \cite{w1}
$$\partial_\nu\varphi(z)= {1\over\varepsilon} U\(z-\xi_\varepsilon\over\varepsilon\){{\langle z-\xi_\varepsilon,\nu\rangle}\over|z-\xi_\varepsilon|}\(1+\mathcal O(\varepsilon)\)
$$
whereas, in the Neumann case
$$
\partial_\nu\varphi(z)= {1\over\varepsilon} U'\(z-\xi_\varepsilon\over\varepsilon\){{\langle z-\xi_\varepsilon,\nu\rangle}\over|z-\xi_\varepsilon|} .
$$
and  by Lemma 8.2 in \cite{w1}
$$\varphi(z)=  -U\(z-\xi_\varepsilon\over\varepsilon\)\(1+\mathcal O(\varepsilon)\) . $$
Then using \eqref{groundstate} and taking into account Remark
\ref{wei_correction1} we  get
\begin{equation}\label{eq:l2}
\begin{split}
\varepsilon^{-N}\int\limits_\Omega \varphi_{\varepsilon,\xi_\varepsilon}^2(x)dx
&= \varepsilon^{2-N}  \int\limits_{\partial\Omega} \frac 1\varepsilon  e^{-{2|z-\xi_\varepsilon|\over\varepsilon}} \left|{z-\xi_\varepsilon\over \varepsilon}\right|^{-(N-1)} (\mathfrak c+o(1))(1+O(\varepsilon)){{\langle z-\xi_\varepsilon,\nu\rangle}\over|z-\xi_\varepsilon|} dz
\\
&=\mathcal O\(|\varphi_{\varepsilon,\xi_\varepsilon} (\xi_\varepsilon)|\).\end{split}
\end{equation}
Then, applying Cauchy-Schwarz inequality and recalling \eqref{restoe},
one deduces that
\[
\begin{split}
\varepsilon^{-N} \int_{\Omega}\varphi_{\varepsilon,\xi_\varepsilon} \phi_{\varepsilon,\xi_\varepsilon}
&\leq \varepsilon^{-N/2}\|\varphi_{\varepsilon,\xi_\varepsilon} \|_{2}|\varphi_{\varepsilon,\xi_\varepsilon}
(\xi_{\varepsilon}) |^{\min\{1,p/2\}}=O(|\varphi_{\varepsilon,\xi_\varepsilon} (\xi_{\varepsilon})|^{\frac12})|\varphi_{\varepsilon,\xi_\varepsilon}
(\xi_{\varepsilon}) |^{\min\{1,p/2\}}
\\
& =o(|\varphi_{\varepsilon,\xi_\varepsilon} (\xi_{\varepsilon})|).
\end{split}
\]
Using this last estimate, together with \eqref{eq:l2}, in \eqref{cru12}
we obtain
$$
\varepsilon^{-{4\over p-1}}
 \int\limits_\Omega u_\varepsilon^2(x)dx
 =
2\sigma_{0} -2\Theta_\varepsilon
+\mathcal O\( |\varphi_{\varepsilon,\xi_\varepsilon} (\xi_\varepsilon)|\).
$$
In order to conclude the proof it is enough to apply
Lemma \ref{lemma_cri1}, and to recall that
$\varphi_{\eps,\xi_\eps}$ (and thus
$\Theta_\eps$) is positive (resp. negative) in the case of Dirichlet (resp. Neumann) boundary conditions (see Proposition \ref{visco}).
\end{proof}

\begin{remark}\label{rem:DirN=1}
Let us consider the case $N=1$. Without loss of generality, we can assume $\Omega=(-1,1).$
A straightforward computation shows that
in the Dirichlet case
\begin{equation}\label{diri}
\varphi_{\varepsilon,0}(x)={U\(\frac1\varepsilon\) \cosh \frac x\varepsilon\over \cosh \frac1\varepsilon}\end{equation}
 and in the Neumann case
\begin{equation}\label{neu}
\varphi_{\varepsilon,0}(x)={U'\(\frac1\varepsilon\) \cosh \frac x\varepsilon\over \sinh \frac1\varepsilon}.
\end{equation}
This is because $\varphi=\varphi_{\varepsilon,0}$ solves
$$-\varepsilon^2\varphi^{\prime\prime}+\varphi=0\ \hbox{in}\ (-1,1)$$
with boundary condition
$$\varphi(1)=\varphi(-1)=U\(1/\varepsilon\)\ \hbox{in the Dirichlet case}$$
or
$$\varphi'(1)={1\over \varepsilon}U'\(1/\varepsilon\),\  \varphi'(-1)=-{1\over \varepsilon}U'\(1/\varepsilon\)\ \hbox{in the Neumann case}.$$

Here  $U$ is explicitly given by
$U(x)=3^{1/4}(\cosh 2x)^{-1/2}.$
In particular
$$\varphi_{\varepsilon,0}(0)\sim \pm  2^{3/2} 3^{1/4}  e^{-{2/\varepsilon}}.$$
Moreover we have
$$\Theta_\varepsilon:=-2\int\limits
  _{-{1\over\varepsilon}}^{1\over\varepsilon}\varphi_{\varepsilon,0}(\varepsilon y )U\(y\)dy \sim\left\{\begin{aligned} - 3^{1/4} 8 {1\over\varepsilon}e^{-{2\over\varepsilon}}\ \hbox{in the Dirichlet case}\\
+3^{1/4} 8 {1\over\varepsilon}e^{-{2\over\varepsilon}}\ \hbox{in the Neumann case},
  \end{aligned}\right.$$
because
$$\int \limits_{-{1\over\varepsilon}}^{1\over\varepsilon}
  \cosh (y) (\cosh 2y)^{-1/2}dy={\sqrt 2}\log\(\sqrt 2\sinh y+\sqrt {2\sinh ^2 y+1}\)\Big |^{y=1/\varepsilon}_{y=0}\sim \sqrt2 {1\over\varepsilon}.
$$
   Finally, the leading term is
  $$\Theta_\varepsilon=-2\int\limits
  _{-{1\over\varepsilon}}^{1\over\varepsilon}\varphi_{\varepsilon,0}(\varepsilon y )U\(y\)dy \sim\left\{\begin{aligned}  - 3^{1/4} 8  {1\over\varepsilon}e^{-{2\over\varepsilon}}\ \hbox{in the Dirichlet case,}\\
 +3^{1/4} 8  {1\over\varepsilon}e^{-{2\over\varepsilon}}\ \hbox{in the Neumann case.}
  \end{aligned}\right.$$
\end{remark}

\begin{remark}\label{rm:morseint}
Let us assume that $\xi_0\in\Omega$ is a non-degenerate peak point (see Definition (1.4)-(1.5) in \cite{w3}) of the
distance function from $\partial\Omega$, i.e. there exists $a\in\mathbb R^N$ such that
$$\int\limits_{\partial\Omega}e^{\langle z-\xi_0,a\rangle}(z-\xi_0)d\mu_{\xi_0}=0$$
and the matrix
$$G(\xi_0):=\left(
\int\limits_{\partial\Omega}e^{\langle z-\xi_0,a\rangle}(z-\xi_0)_i(z-\xi_0)_jd\mu_{\xi_0}\right)_{i,j=1,\dots,N}\ \hbox{is non-singular},$$
In particular, all its eigenvalues are strictly positive. We remark that if  $\Omega$ is a ball then its center is a non-degenerate peak point.
Combining results in \cite{MR1690196,w3}, we  get that if $\epsilon$ is small enough the (unique) solution to the Dirichlet or the Neumann problem which concentrates at $\xi_0$ is non-degenerate and its Morse index is equal to $1$ in the Dirichlet case (Theorem 6.2 in \cite{MR1690196})
and is equal to $1+N$ in the Neumann case (see Theorem 1.3   in \cite{w3}).
\end{remark}

  \section{The Schr\"{o}dinger equation}\label{sub:sc}
In this section we will tackle problem \eqref{P0} for $\Omega=\R^{N}$.

First of all let us solve the singularly perturbed Schr\"{o}dinger equation
\begin{equation}\label{sch}
-\varepsilon^2\Delta u+\( \varepsilon^2 V(x)+1\) u = u^p\ \hbox{in}\ \mathbb R^N,\ u>0\ \hbox{in}\ \mathbb R^N.
\end{equation}
For sake of simplicity we will assume  $V, |\nabla V| \in L^{\infty}(\mathbb R^N)$ and, given a non-degenerate critical point $\xi_0$ of $V$, we suppose that   in a neighbourhood of $\xi_0$ the following  expansion holds true:
 \begin{equation}\label{taylor}
 V(x)=\sum\limits_{i=1}^N a_i \(x -{\xi_0}\)^2+\mathcal O\(|x-\xi_0|^3\),\ \hbox{where}\ a_i\not=0.
 \end{equation}

 The following result can be easily proved by a Ljapunov-Schmidt procedure combining the ideas of Li \cite{yyl2}, Grossi \cite{g}
 and Grossi and Pistoia \cite{gp}. A sketch of the proof is given in the Appendix.

\begin{proposition}\label{sc_exi}
 Let $\xi_0$ be a non-degenerate critical point of $V $.
  There exists $\varepsilon_0>0$ such that for any $\varepsilon\in(0,\varepsilon _0)$ there exists a solution $u_\varepsilon$ to \eqref{sch}
 which concentrate at the point $\xi_0 $   as $\varepsilon\to0.$ More precisely,
\begin{equation}\label{usc}u_\varepsilon(x)= U\(\frac{x-\xi_\varepsilon}\varepsilon\)
 -\varepsilon^4W_{\xi_0}\(\frac{x-\xi_\varepsilon}\varepsilon\)+\phi_{\varepsilon} (x)\end{equation}
 where
\begin{equation}\label{pointsc} \xi_\varepsilon\to \xi_0 \ \hbox{as}\ \varepsilon\to0,\end{equation}
 the function $W_{\xi_0}\in H^1(\mathbb R^N)$
 solves the linear problem
 \begin{equation}\label{W}-\Delta W_{\xi_0}+W_{\xi_0}-pU^{p-1}W_{\xi_0}=\sum\limits_{i=1}^N a_i y_i^2U(y)\ \hbox{in}\ \mathbb R^N\end{equation}
  and the remainder term $\phi_\varepsilon$ satisfies
\begin{equation}\label{restosc} \|\phi_{\varepsilon }\| _{H^1_\varepsilon (\mathbb R^N)}:=\(\int\limits_{\mathbb R^N}\( \varepsilon^2|\nabla \phi_{\varepsilon }|^2+ \phi_{\varepsilon }^2\)dx\)^{1/2}=
 \mathcal O\(\varepsilon^{{N\over2}+4+\eta}\)\
 \hbox{for some}\ \eta>0.
\end{equation}
 \end{proposition}
 Next, we consider the  Schr\"{o}dinger equation with prescribed $L^2-$norm
\begin{equation}\label{schp}
\begin{cases}
-\varepsilon^2\Delta u+\( \varepsilon^2 V(x)+1\) u = u^p\ \hbox{in}\ \mathbb R^N,
\\
u>0\ \hbox{in}\ \mathbb R^N,
\\
\varepsilon^{-{4\over p-1} }\int\limits_{\mathbb R^N} u^2=\rho.
\end{cases}
\end{equation}
We  will first give an existence result
in the non-critical case.
\begin{theorem}\label{main3}
 Let $\xi_0\in\R^N$ be a non-degenerate critical point of $V$. Suppose that $p\neq \frac4N +1$ and take
 $\sigma_{0}$ as in \eqref{so}. The following conclusions
 hold
  \begin{itemize}
 \item[(i)] If $p<\frac 4N+1$ there exists $R>0$ such that for any $\rho>R$ problem \eqref{pnl} has a solution $(u_\rho,\Lambda_\rho)$ for
 $\varepsilon:=\(\Lambda_\rho\rho\)^{(p-1)\over (p-1)N-4}$  with $\Lambda_\rho \to {1\over2\sigma_0}$ and $u_\rho$   concentrating at the point $\xi_0$ as $\rho\to\infty$.
 \item[(ii)]  If $p>\frac 4N+1$ there exists $r>0$ such that for any $\rho<r$ problem \eqref{pnl}  has a solution $(u_\rho,\Lambda_\rho)$
 for $\varepsilon:=\(\Lambda_\rho\rho\)^{(p-1)\over (p-1)N-4}$ with $\Lambda_\rho \to {1\over 2\sigma_0}$ and $u_\rho$   concentrating at the point $\xi_0$ as $\rho\to0$.
\end{itemize}
 \end{theorem}

  \begin{proof}
Following the same argument of the previous sections we  reduce the existence of solutions to problem \eqref{schp} with variable but prescribed $L^2-$norm to the existence of solutions to problem \eqref{sch} where the parameter $\varepsilon$ is small.
Let us choose
 \begin{equation}\label{lasc}
 \varepsilon^{-{4\over p-1}+  N}=\Lambda \rho\ \hbox{with}\ \Lambda=\Lambda(\rho)\in\left[\frac{1}{ 2\sigma_0} ,\frac2{ 2\sigma_0}\right]
 \end{equation}
 where  $\sigma_0$ is defined in \eqref{so}.
 It is clear that $\varepsilon\to0$ if and only if either
  $ p<\frac 4N +1$ and $\rho\to\infty$ or
  $  p>\frac 4N +1$ and $\rho\to0.$
 By   Proposition \ref{sc_exi} we deduce   that for any $\Lambda$ as in \eqref{la}, there exists either $R>0$ or $r>0$ such that for any $\rho>R$ or $\rho<r$ problem \eqref{sch} has a solution $u_{\varepsilon}$ as in \eqref{usc} such that $\varepsilon$ satisfies \eqref{lasc}.
 Now, we have to choose the free parameter $\Lambda=\Lambda(\rho)$ such that the $L^2-$norm of the solution is the  prescribed value.
By \eqref{restosc} we deduce
 \begin{equation}\label{cru1sc}
 \begin{aligned}
 \varepsilon^{-{4\over p-1}} \int\limits_{\mathbb R^N}u_{ \varepsilon}
^2(x)dx&=  \varepsilon^{-{4\over p-1}}  \int\limits_{\mathbb R^N}\( U \({x-
\xi_\varepsilon\over\varepsilon}\)-\varepsilon^4W_{\xi_0} \({x-
\xi_\varepsilon\over\varepsilon}\)+\phi_\varepsilon(x)\)^2dx
 \\&
 =   \varepsilon^{-{4\over p-1}+N} \left[ \int\limits_{\mathbb R^N} U^2 \(y\) dy+\mathcal O(\varepsilon^4)\right]
\\
&= \varepsilon^{-{4\over p-1}+N}\left[  2 \sigma_0+\mathcal O(\varepsilon^{4})\right]
=\rho\Lambda(\rho)\left[2 \sigma_0+\mathcal O(\varepsilon^{4})\right],
\\
\end{aligned}\end{equation}
where  the term $\mathcal O(\varepsilon^4)$   is uniform with respect to $\Lambda=\Lambda(\rho)$   when either $\rho\to+\infty$ or $\rho\to0$
and the last equality comes from  \eqref{lasc}.

Finally, it is clear that it is possible to choose $\Lambda(\rho)$ satisfying \eqref{lasc}, when either $\rho\to+\infty$ or $\rho\to0,$ such that
 $ \Lambda=\frac{1}{2\sigma_0}+o(1) $, implying that
$u_\varepsilon$ has the prescribed $L^2-$norm. That concludes the proof.\end{proof}
The result in the mass critical case requires an extra assumption. Before stating it, it is useful to point out the following fact.
\begin{remark}\label{ipo0}

Let us point out that $W_{\xi_0}$ can be written as
 \begin{equation}\label{w1}
 W_{\xi_0}(y)=\sum\limits_{i=1}^N a_i W_i(y),\end{equation}
 where each $W_i$ solves
solves
 $$-\Delta W_i+W_i-pU^{p-1}W_i=  y_i^2 U(y).$$
  If $W_1$ denotes the solution to
$$-\Delta W_1+W_1-pU^{p-1}W_1=  y_1^2 U(y),$$
it is clear that
$$W_i(y_1,\dots,y_i,\dots,y_N):=W_1(y_i,\dots,y_1,\dots,y_N)$$
Therefore
$$\begin{aligned}\int\limits_{\mathbb R^N} W_{\xi_0}(y)U(y)dy&=\sum\limits_{i=1}^N a_i \int\limits_{\mathbb R^N} W_i(y)U(y)dy\\
&=\frac 1N\sum\limits_{i=1}^N a_i \int\limits_{\mathbb R^N} \underbrace{\(W_1(y)+\dots+W_N(y)\)}_{:=W(y)}U(y)dy\\
&=2\sum\limits_{i=1}^N a_i \underbrace{\frac 1{2N}\int\limits_{\mathbb R^N} W(y)U(y)dy}_{:=\mathfrak m}=\mathfrak m\Delta V(\xi_0),\\
\end{aligned}$$
where $W$ solves
 \begin{equation}\label{ipo2}
 -\Delta W +W -pU^{p-1}W =  |y|^2 U(y)\ \hbox{in}\ \mathbb R^N.\end{equation}
 \end{remark}
\begin{theorem}\label{th:main3crit}
 Let $p=\frac 4N+1$, $\sigma_{0}$ as in \eqref{so} and $\xi_0\in\R^N$ be a non-degenerate critical point of $V$  such that   $\Delta V(\xi_0)\not=0$. Assume
 \begin{equation}\label{ipo}
 {\mathfrak m} :=\frac 1{2N}\int\limits_{\mathbb R^N} U (y) W \(y\)dy\not=0
 \end{equation}
 where $W$ is defined in \eqref{ipo2}.  there exists $\delta>0$ such that if
 either ${\mathfrak m}\Delta V(\xi_0)>0$ and $\rho\in\(2\sigma_0-\delta,2\sigma_0\)$ or ${\mathfrak m}\Delta V(\xi_0)<0$ and $\rho\in\(2\sigma_0,2\sigma_0+\delta\)$
 problem \eqref{schp}
 with $\varepsilon^4:= \Lambda_\rho|\rho-2\sigma_0|$ has a solution $(u_\rho,\Lambda_\rho)$ such that $\Lambda_\rho \to {1\over |\mathfrak m\Delta V(\xi_0)| }$  and $u_\rho$ concentrates at the point $\xi_0$ as $\rho\to2\sigma_0.$
\end{theorem}
\begin{proof}
In this case, we need a more refined profile of the solution $u_\varepsilon$, namely the first order  expansion $W_{\xi_0}$ given in \eqref{W} of the remainder term (see also Remark \eqref{ipo0}). Let us choose
\begin{equation}\label{la2sc}
\varepsilon^4=\Lambda\delta\ \hbox{where}\ \delta:=|\rho-2\sigma_0| \ \hbox{and}\ \Lambda=\Lambda(\delta)\in\left[
\frac1{2{\mathfrak m}\Delta V(\xi_0) },\frac 2{{\mathfrak m}\Delta V(\xi_0) }\right].
\end{equation}

Now, we have to choose the free parameter $\Lambda=\Lambda(\delta)$ such that   the $L^2-$norm of the solution is the  prescribed value.
Equation \eqref{cru1sc} becomes
 \begin{equation}\label{cru2sc}
 \begin{aligned}
 \varepsilon^{-{4\over p-1}} \int\limits_{\mathbb R^N}u_{ \varepsilon}^2(x)dx
 &
 =
 2 \sigma_0 -2\varepsilon^4  {\mathfrak m}\Delta V(\xi_0)
  +\mathcal O\(\varepsilon^{4+\eta}\) \\
&
= \rho\pm\delta-2\delta\Lambda(\delta){\mathfrak m}\Delta V(\xi_0)+o\(\delta\),\\
 \end{aligned}\end{equation}
 where the term $o(\cdot)$   is uniform with respect to $\Lambda=\Lambda(\delta)$ and where the last equality comes from
 \eqref{la2sc}.

In order to conclude the proof it is enough to choose $\Lambda(\delta)$
satisfying  \eqref{la2sc}, for $\delta\to0,$ such that
 $$
 \delta\(1 +{\mathfrak m}\Delta V(\xi_0) \Lambda(\delta)+o\(1\)\)=0,\quad
 \hbox{or}\quad \delta\(-1 +{\mathfrak m}\Delta V(\xi_0)\Lambda(\delta)+o\(1\)\)=0
$$
 (in particular ${\mathfrak m}\Delta V(\xi_0)<0$ in the first case and ${\mathfrak m}\Delta V(\xi_0)>0$ in the second case)
and  by \eqref{cru2sc} we deduce that $u_\varepsilon$ has the prescribed $L^2-$norm. That concludes the proof.\\
\end{proof}

In the following  remark   we prove that $ {\mathfrak m} >0$ and so \eqref{ipo}  is true when $N=1$ as proved.
We conjecture that this is true in any dimension.

\begin{remark}\label{N=1}
If $N=1$ then
$ {\mathfrak m} >0.$  In particular, assumption \eqref{ipo}  holds true and
\begin{itemize}
\item[(i)] if $\xi_0$ is a non-degenerate minimum point of $V$ then
$\mathfrak m V''(\xi_0)>0$
\item[(ii)] if $\xi_0$ is a non-degenerate maximum point of $V$ then
$\mathfrak m V''(\xi_0)<0.$
\end{itemize}
First of all,  we remark that when $N=1$, $U$ is explicitly given by
$U(x)=3^{1/4}(\cosh 2x)^{-1/2}$. Moreover,
$W_{\xi_0}=V\rq{}\rq{}(\xi_0) W,$ where $W\in H^1(\mathbb R)$  solves
\begin{equation}\label{eq:crit}
 -W\rq{}\rq{} +W-pU^{p-1}W=y^2 U(y)\ \hbox{in}\  \mathbb R.\end{equation}
We look for an even  solution to \eqref{eq:crit} of the form $W(r)=c(r)U\rq{}(r)$ and we take into account that $U'$ solves
$-(U')''+U'-pU^{p-1}U'=0$ to obtain
 that $c(r)$ has to satisfy the equation
$$
-c\rq{}\rq{}U\rq{}-2c\rq{}U\rq{}\rq{} =r^2U\ \hbox{if}\ r>0.
$$
Multiplying by $ U\rq{}$, we get
$$
-\left(c\rq{}(U\rq{}(r))^2 \right)\rq{}=\frac12r^{2} (U^2(r))\rq{}
$$
yielding
$$
c'(r)(U'(r))^{2}-c'(t)(U'(t))^{2}=\int\limits_r^t \frac12s^{2} (U^2(s))\rq{} ds>0\qquad \text{for}\ 0<r<t<\infty.
$$
 Notice that $r\to r^{2} (U^2(r))\rq{}$ is an $L^{2}(\mathbb R)-$function and so
  $r\to c\rq{}(r)(U\rq{}(r))^2$ is an $H^{1}(\mathbb R)-$function, which implies    that
$c\rq{}(t)(U\rq{}(t))^2\to 0$ as $t\to \infty$.
Then, we   get
 $$
 c'(r)=\dfrac1{2(U'(r))^{2}}\int\limits_r^{\infty} s^{2} (U^2(s))\rq{}ds \qquad \hbox{if}\ r>0.
 $$
In order to compute $\displaystyle  \lim_{r\to +\infty}c'(r)$ we notice that we are in the position to apply de L'Hopital rule and we obtain
\[
\lim_{r\to +\infty}c'(r)= \lim_{r\to +\infty}\dfrac{-r^{2}U(r)U'(r)}{2U'U''}
= \lim_{r\to +\infty}\dfrac{-r^{2}U(r)}{2U''(r)}=-\infty
\quad\text{as} \lim_{r\to +\infty}\dfrac{U(r)}{U''(r)}=1.
\]
The previous computation also yields
\[
\lim_{r\to +\infty}\frac{c'(r)}{-\frac{r^{2}}2}=1.
\]
In addition, since $U'(r)/r \to U''(0)\neq 0$ as $r\to0^+$,
\[
\lim_{r\to0^{+}}rc'(r)=\frac12\lim_{r\to0^{+}}\frac{r^2}{[U'(r)]^{2}}\int_{0}^{+\infty}s^2U(s)U'(s)ds=-\infty.
\]
This immediately implies that
\[
\lim_{r\to0^{+}}c(r)=+\infty,
\]
and (again using de L'Hopital rule)
\begin{align*}
\lim_{r\to0^{+}}W(r)
&
=\lim_{r\to0^{+}}c(r)U'(r)=\lim_{r\to0^{+}}
\frac{c(r)}{\frac1{U'(r)}}=\lim_{r\to0^{+}}\frac{\frac{1}{[U'(r)]^{2}}\int_{r}^{+\infty}s^2U(s)U'(s)ds}{-\frac{U''(r)}{[U'(r)]^{2}}}
\\ &=\lim_{r\to0^{+}}\frac{\int_{r}^{+\infty}s^2U(s)U'(s)ds}{-U''(r)}
 =\frac{\int_{0}^{+\infty}s^2U(s)U'(s)ds}{-U''(0)}=-\frac{3^{1/4}G}{4} = -0.301...,
\end{align*}
where $G$ is the Catalan constant:
\[
G = \frac12 \int_0^{+\infty} \frac{t}{\cosh t}dt = 0.916...\ .
\]
The above consideration imply that $W$ is the unique solution of the following Cauchy problem
\[
\begin{cases}
-W''+(1-pU^{p-1})W=r^{2}U
\\
W(0)=-\frac{3^{1/4}}{4} G
\\
 W'(0)=0.
\end{cases}
\]
Since $c$ is monotone, we deduce that $W$ has exactly one zero $r_0$, and it is possible to show that $0<r_0<1$. As a consequence
\[
\int_{0}^{+\infty} U(r)W(r)dr>\int_{0}^{2} U(r)W(r)dr \approx 0.253688...>0
\]
(by continuous dependence, the above integral can be numerically estimated at any level of accuracy).
\end{remark}

\begin{remark}\label{rm:morsesc}
We point out that  if $\xi_0$ is a non-degenerate critical  point of the $V$ whose Morse  index  is $m(\xi_0)$
then by Corollary 1.2 in \cite{grse} we deduce that
the solution   concentrating at a $\xi_0$ is non-degenerate and has Morse index $1+m(\xi_0)$. In particular, the solution  concentrating at a non-degenerate minimum point of $V$ is non-degenerate and has Morse index 1.
\end{remark}

\section{Appendix}
Let us briefly sketch the proof of Proposition \ref{sc_exi}. Let us introduce some notations.
Let $H^1(\mathbb R^N)$ be equipped with the usual scalar product and norm
$$\langle u,v\rangle=\int\limits_{\mathbb R^N}(\nabla u\nabla v+uv)dx\ \hbox{and}\
\|u\|=\(\int\limits_{\mathbb R^N}(|\nabla u|^2+u^2)dx\)^{1/2}.$$
We know that the embedding $H^1(\mathbb R^N)\hookrightarrow L^2(\mathbb R^N)$ is continuous. Let $i^*:L^2(\mathbb R^N)\rightarrow H^1(\mathbb R^N)$ be the adjoint defined by
$$u=i^*(f)\ \hbox{if and only if $u\in H^1(\mathbb R^N)$ solves}\ -\Delta u+u=f\ \hbox{in}\ \mathbb R^N.$$
We point out that
\begin{equation}\label{is}
\|i^*(f)\|\le \|f\|_{L^2(\mathbb R^N)}\ \hbox{for any}\ f\in L^2(\mathbb R^N).
\end{equation}
Now, let us remark that if $\xi\in\mathbb R^N$ and $v(x):=u(\varepsilon x+\xi)$ then $u$ solves equation \eqref{sch} if and only if $v$ solves the equation
 $$  - \Delta v+\( \varepsilon^2 V(\varepsilon x+\xi)+1\) v = v^p\ \hbox{in}\ \mathbb R^N,\ v>0\ \hbox{in}\ \mathbb R^N,
$$
which can be rewritten as
\begin{equation}\label{sc1}
v=i^*\(f(v)-\varepsilon^2 V_{\varepsilon,\tau} v \), \hbox{where}\ f(v):=(v^+)^p\ \hbox{and}\ V_{\varepsilon,\tau}(x) :=V(\varepsilon x+\varepsilon^2 \tau+\xi_0) ,
\end{equation}
  where we choose the point $\xi$ as
    \begin{equation}\label{sc20}
   \xi=\varepsilon ^2\tau+\xi_0\ \hbox{with}\ \tau\in\mathbb R^N.
  \end{equation}
  Let us look for a solution to \eqref{sc1} of the form
  \begin{equation}\label{sc2}
  v(x)=Z(x)+\phi(x),\ \hbox{where}\ Z(x):=U(x)-\varepsilon^4 W_{\xi_0}(x),
  \end{equation}
 $U$ is the radial solution to \eqref{pblim}
and $W_{\xi_0}\in K^\perp$ is an exponentially decaying solution to the linear problem
$$-\Delta W_{\xi_0}+W_{\xi_0}-pU^{p-1}W_{\xi_0}=H_{\xi_0},\ H_{\xi_0}(y):=\sum\limits_{i=1}^N a_i y_i^2 U(y)\ \hbox{in}\ \mathbb R^N
$$
  and $\phi$ is a remainder term which belongs to the space
 $$K^\perp:=\left\{\phi\in H^1(\mathbb R^N)\ :\ \langle \phi, \partial_i U\rangle=0,\ i=1,\dots,N\right\},$$
which is orthogonal, with respect to the $H^{1}(\R^{N})$ norm,
to the $N-$dimensional space
$$
K:=span\ \{\partial_1 U,\dots,\partial _N U\},$$
formed by the solutions to the linear equation
 $$-\Delta \psi+\psi-pU^{p-1}\psi=0\ \hbox{in}\ \mathbb R^N.$$
Problem \eqref{sc1} can be rewritten as
\begin{equation}\label{sc11}
\begin{aligned}
\underbrace{\phi-i^*\left\{\left[f'(Z)-\varepsilon^2 V_{\varepsilon,\tau}\right]\phi\right\}}_{:=\mathcal L_{\varepsilon,\tau}(\phi)}&=\underbrace{i^*\left\{f(Z+\phi)-f(Z)-f'(Z)\phi\right\}}_{:=\mathcal N_{\varepsilon,\tau}(\phi)}\\
&\underbrace{+i^*\left\{f(Z)-\varepsilon^2 V_{\varepsilon,\tau}Z\right\}-Z}_{:=\mathcal E_{\varepsilon,\tau}}.\end{aligned}
\end{equation}

Let us denote by $\Pi:H^1(\mathbb R^N)\to K$ and $\Pi^\perp:H^1(\mathbb R^N)\to K^\perp$ the orthogonal projections.
Then, problem \eqref{sc11} turns out to be equivalent to the system
\begin{equation}\label{sc31}
\Pi^\perp\left\{\mathcal L_{\varepsilon,\tau}(\phi)-\mathcal N_{\varepsilon,\tau}(\phi)-\mathcal E _{\varepsilon,\tau}\right\}=0
\end{equation}
and
\begin{equation}\label{sc32}
\Pi\left\{\mathcal L_{\varepsilon,\tau}(\phi)-\mathcal N_{\varepsilon,\tau}(\phi)-\mathcal E _{\varepsilon,\tau}\right\}=0.
\end{equation}

  First, for $\varepsilon$ small and for any $\xi\in \mathbb R^N$ we will find a solution $\phi=\phi_{\varepsilon,\tau}\in K^\perp$ to \eqref{sc31}.
We recall that we are assuming, for the sake of simplicity, that $V$ and
$\ |\nabla V|$ are $ L^\infty(\mathbb R^N)$ function.
  \begin{proposition}
  \label{phi}
   For any compact set $T\subset\mathbb R^N$ there exists $\varepsilon_0>0$ and $C>0$ such that  for any $\varepsilon\in (0,\varepsilon_0)$ and for any $\tau\in T$ there exists a unique $\phi=\phi_{\varepsilon,\tau}\in K^\perp$ which solves equation \eqref{sc31} and
  $$\|\phi_{\varepsilon,\tau} \|\le C\varepsilon^5.$$
  \end{proposition}
  \begin{proof}
  Let us sketch the main steps of the proof.
 \begin{itemize}
 \item[(i)] First of all , we prove that the linear operator $\mathcal L_{\varepsilon,\tau}$ is uniformly invertible in $K^\perp,$ namely
  there exists $\varepsilon_0>0$ and $C>0$ such that
$$
\|\mathcal L_{\varepsilon,\tau}(\phi)\|
\geq 
C\|\phi\|\ \hbox{for any}\ \varepsilon\in (0,\varepsilon_0),\ \tau\in T\ \hbox{and}\ \phi\in K^\perp.
$$
We can argue as in  \cite{g,gp}.
\item[(ii)] Next, we compute the size of the error $\mathcal E_{\varepsilon,\tau}$ in terms of $\varepsilon.$ More precisely, we show that
 there exists $\varepsilon_0>0$ and $C>0$ such that
  $$\|\mathcal E_{\varepsilon,\tau} \|\le C\varepsilon^5\ \hbox{for any}\ \varepsilon\in (0,\varepsilon_0)\ \hbox{and}\ \tau\in T.$$
Indeed, we recall that
  $$Z=U-\varepsilon^4W_{\xi_0}=i^*\left\{f(U)-\varepsilon^4 \left[H_{\xi_0}+f'(U)W_{\xi_0}\right]\right\}.$$
  Moreover by \eqref{taylor} we deduce
  $$V_{\varepsilon,\tau}(x)=V(\varepsilon x+\varepsilon ^2\tau+\xi_0)=\varepsilon^2\sum\limits_{i=1}^N a_i x_i^2+\mathcal O(\varepsilon^3 \(1+|x|^3\)).$$
 Therefore we have
  $$\begin{aligned}&i^*\left\{f(Z)-\varepsilon^2 V_{\varepsilon,\tau}Z\right\}-Z\\
  &=
  i^*\left\{f(U-\varepsilon^4 W_{\xi_0})-\varepsilon^2 \left[\varepsilon^2 \sum\limits_{i=1}^N a_i x_i^2+\mathcal O\(\varepsilon^3\(1+ |x|^3\)\)\right]\left[U-\varepsilon^4 W_{\xi_0}\right]\right.\\ &\qquad \left.
  -f(U)+\varepsilon^4 \left[H_{\xi_0}+f'(U)W_{\xi_0}\right]\right\}\\
  &= i^*\left\{f(U-\varepsilon^4W_{\xi_0})-f(U)+ \varepsilon^4  f'(U)W_{\xi_0}\right\}\\
  &
 +i^*\left\{\mathcal O\(\varepsilon ^5\(1+ |x|^3\) U+\varepsilon^8|W_{\xi_0}|+ \varepsilon^9\(1+ |x|^3\) |W_{\xi_0}|\)\right\}\\
  \end{aligned}
  $$
  and by \eqref{is} and \eqref{li} we immediately get the claim.\\
  We recall the useful estimate
  \begin{equation}\label{li}
  |f(a+b)-f(a)-f'(a)b|=\left\{\begin{aligned}&\mathcal O(|b|^p)\ \hbox{if}\ 1<p\le2,\\
  &\mathcal O(|b|^p+|a|^{p-2}|b|^2)\ \hbox{if}\ p\ge2.\end{aligned}\right.
  \end{equation}
\item[(iii)] Finally, we use   a standard contraction mapping argument, combined to the fact that the term $\mathcal N_{\varepsilon,\tau}(\phi)$ is super-linear in $\phi$ in virtue of \eqref{li}.
 \end{itemize}
  \end{proof}

    Now, for $\varepsilon$ small enough we fill find a point $\tau_\varepsilon\in\mathbb R^N$ so that
  \eqref{sc32} is also satisfied. That will conclude the proof.
\begin{proposition}\label{punti}
 There exists $\varepsilon_0>0$ such that for any  $\varepsilon\in (0,\varepsilon_0)$ there exists $\tau_\varepsilon\in\mathbb R^N$ such that
 equation \eqref{sc32} is satisfied.
  \end{proposition}
\begin{proof}
Since \eqref{sc31} holds we deduce that there exist real numbers $c^i_{\varepsilon,\tau}$ such that
\begin{equation}\label{pu1}\mathcal L_{\varepsilon,\tau}(\phi_{\varepsilon,\tau})-\mathcal N_{\varepsilon,\tau}(\phi_{\varepsilon,\tau})-\mathcal E _{\varepsilon,\tau}=\sum\limits_{i=1}^N c^i_{\varepsilon,\tau}\partial_i U.\end{equation}
We are going to find points $\tau=\tau_\varepsilon$ such that the $c^i_{\varepsilon,\tau}$'s are zero. \\
Let us multiply \eqref{pu1} by $\partial_j U=i^*\(f'(U)\partial_j U\)$.
We get
\begin{equation}\label{pu2}\left\langle\mathcal L_{\varepsilon,\tau}(\phi_{\varepsilon,\tau})-\mathcal N_{\varepsilon,\tau}(\phi_{\varepsilon,\tau})-\mathcal E _{\varepsilon,\tau},\partial _jU\right\rangle=A c^j_{\varepsilon,\tau},\end{equation}
because
$$\langle\partial_i U,\partial_j U\rangle=\int\limits_{\mathbb R^N} f'(U)\partial_i U\partial_j U= A \delta_{ij },\ \hbox{where}\
A:=\int\limits_{\mathbb R^N} f'(U)\(\partial_1 U\)^2.$$
Moreover, by \eqref{li} we have
$$\langle \mathcal L_{\varepsilon,\tau}(\phi_{\varepsilon,\tau}),\partial_j U\rangle=\int\limits_{\mathbb R^N} \left[f'(U)-f'(U-\varepsilon^4 W_{\xi_0})+\varepsilon^2 V_{\varepsilon,\tau}\right]\phi\partial_j U=\mathcal O\(\varepsilon^7\)$$
and
$$\langle \mathcal N_{\varepsilon,\tau}(\phi_{\varepsilon,\tau}),\partial_j U\rangle= \mathcal O\(\varepsilon^8\).$$
It remains to compute
$$\begin{aligned} -\langle \mathcal E _{\varepsilon,\tau},\partial_j U\rangle&=-\langle i^*\left[f(Z)-\varepsilon^2 V_{\varepsilon,\tau}Z\right]-Z,\partial_j U\rangle\\
&=-\int\limits_{\mathbb R^N}\left[f(Z)-\varepsilon^2 V_{\varepsilon,\tau}Z\right] \partial_j U+
\int\limits_{\mathbb R^N} Zf'(U)   \partial_ jU\\
&=\varepsilon^2\int\limits_{\mathbb R^N} V_{\varepsilon,\tau}Z\partial_j U\ \hbox{\small (indeed $Z=U-\varepsilon^4W_{\xi_0}$ is even, Rem. \eqref{w_even}, and $\partial_j U$ is odd)}\\
&=\varepsilon^2\int\limits_{\mathbb R^N} V(\varepsilon x+\varepsilon^2\tau+\xi_0)(U-\varepsilon^4W_{\xi_0})\partial_j U\\
&=\varepsilon^2\int\limits_{\mathbb R^N} V(\varepsilon x+\varepsilon^2\tau+\xi_0)U\partial_j U+\mathcal O(\varepsilon^6)\\
&=-\frac 12\varepsilon^3\int\limits_{\mathbb R^N} {\partial V\over \partial y_j}(\varepsilon x+\varepsilon^2\tau+\xi_0)U^2(x)dx+\mathcal O(\varepsilon^6)\\
&=-\frac 12\varepsilon^5\left[  a_j  \tau_j\int\limits_{\mathbb R^N} U^2(x)dx+\frac1{2N}\sum\limits_{\ell,\kappa=1}^N{\partial^3 V\over \partial y_\kappa\partial y_\ell\partial y_j} (\xi_0)
 \int\limits_{\mathbb R^N}|x|^2 U^2(x)dx  \right]\\ &+\mathcal O(\varepsilon^6),\\
\end{aligned}$$
because by \eqref{taylor} and by the mean value theorem
$$\begin{aligned}({\partial_j V})(\varepsilon x+\varepsilon^2\tau+\xi_0)&= a_j
 \(\varepsilon x_j+\varepsilon^2\tau_j\)+\frac12\sum\limits_{\ell,\kappa=1}^N{\partial^3 V\over \partial y_\kappa\partial y_\ell\partial y_j} (\xi_0)
  \(\varepsilon^2x_\ell x_\kappa\)\\ &+\mathcal O\(\varepsilon^3\(1+|x|^3\)\).\end{aligned}$$
  Therefore,   \eqref{pu2} reads as the system
  $$-\frac 12\varepsilon^5\left[ B a_j \tau_j+C\sum\limits_{\ell,\kappa=1}^N{\partial^3 V\over \partial y_\kappa\partial y_\ell\partial y_j} (\xi_0)+o(1)\right]=A c^{j}_{\varepsilon,\tau}\ \hbox{for any}\ j=1,\dots,N,$$
  for some positive constants $A,$ $B$ and $C$.
  Finally, since all the $a_j$'s are different from zero, if $\varepsilon$ is small enough there exists $\tau=\tau_\varepsilon$ such that the R.H.S is zero and so all the $c^j_{\varepsilon,\tau_\varepsilon}$'s are zero.
\end{proof}
  \begin{remark}\label{w_even}
Let us point out that $W_{\xi_0}$ is even in each $y_i$'s.
By \eqref{w1} it is enough to prove that $W_1\in K^\perp$ which solves
 $$-\Delta W_1+W_1-pU^{p-1}W_1=  y_1^2 U(y)\ \hbox{in}\ \mathbb R^N$$
 is even in $y_1,$ i.e.
$W(y_1,y')=W(-y_1,y')$ where $y'=(y_2,\dots,y_N).$
It is immediate to check that the function
$$w(y)=W(y_1,y')-W(-y_1,y')=\sum\limits_{i=1}^N \omega_i \partial_i U=\frac{U'(|y|)}{|y|}\sum\limits_{i=1}^N \omega_i y_i$$
where $\rho=|y|$, for some $\omega_i\in\mathbb R$, since it solves the linear equation
 $$-\Delta w+w-pU^{p-1}w=  0.$$
It is clear that $\omega_2=\dots=\omega_N=0$ and so $w(y)=\frac{U'(|y|)}{|y|} \omega_1 y_1.$
Now, by the orthogonality condition we deduce
$$\begin{aligned}0&=\langle W_1,\partial_1 U\rangle=\int\limits_{\mathbb R^N}pU^{p-1}\partial_1 U W_1\\ &=
\int\limits_{\{y_1\ge0\}}pU^{p-1}(|y|)\frac{U'(|y|)}{|y|} y_1 W_1(y_1,y')dy+\int\limits_{\{y_1\le0\}}pU^{p-1}(|y|)\frac{U'(|y|)}{|y|} y_1 W_1(y_1,y')dy=\\
&=
\int\limits_{\{y_1\ge0\}}pU^{p-1}(|y|)\frac{U'(|y|)}{|y|} y_1 \underbrace{\left[W_1(y_1,y')-W_1(-y_1,y')\right]}_{=w(y)}dy=\\
&=\omega_1\int\limits_{\{y_1\ge0\}}pU^{p-1}(|y|)\(\frac{U'(|y|)}{|y|} y_1\)^2dy,
\end{aligned}$$
which implies $\omega_1=0$. That concludes the proof.
\end{remark}

\bibliography{normalized}
\bibliographystyle{abbrv}

\end{document}